\documentclass[11pt]{article}
\usepackage{anysize}
\usepackage{graphicx}
\usepackage{amsthm}
\usepackage[utf8]{inputenc}
\usepackage{amsfonts}
\usepackage{amssymb}
\usepackage{amsmath}
\usepackage[spanish]{babel}
\newtheorem{defi}{Definici\'on}

\newtheorem{theo}{Teorema}
\newtheorem{prop}{Proposici\'on}
\newtheorem{lema}{Lema}

\newcommand{\arc}[1]{\ensuremath{\overset{\frown}{\raisebox{0pt}[6pt]{#1}}}}
\marginsize{3cm}{3cm}{2cm}{2cm}
\begin{document}
\title{?`Podemos aceptar la definici\'on usual de las funciones trigonom\'etricas?}
\author{{\bf \large Ignacio Tejeda}\\ [0.1cm]
Facultad de Matem\'aticas\\ [0.05cm]
Pontificia Universidad Cat\'olica de Chile}
\maketitle

\begin{abstract}
Seg\'un diversos autores, un \'angulo es un objeto geom\'etrico formado en una superficie plana por una recta $\ell$ y otra recta $\ell'$ que la intersecta en alg\'un punto $A$. Si sobre tal superficie se trazan los ejes ortogonales $X$,$Y$, de modo que $X$ coincida con $\ell$ y el origen est\'e situado en $A$, se puede identificar el \'angulo $\angle A$ con el punto de intersecci\'on entre la recta $\ell'$ y la circunferencia de centro $A$ y radio $1$, lo que permite definir el seno de $\angle A$ como la ordenada del punto con que este \'angulo se identifica. Sin embargo, esta definici\'on s\'olo da significado al seno del objeto geom\'etrico $\angle A$, pero no al seno de un n\'umero real.\\
\\
La deficiencia anterior se resuelve si a cada punto $P$ de la circunferencia unitaria se le asigna la longitud $\varphi(P)$ del arco que va de $(1,0)$ a $P$, pues de este modo la expresi\'on $\sen(x)$ puede definirse como la ordenada del punto $\varphi^{-1}(x)$. Esto es intuitivamente adecuado, pero da pie a algunas interrogantes:\\
\\
1.\ ?`Qu\'e entendemos por longitud de un arco de circunferencia?\\
\\
Para simplificar ideas consideraremos la porci\'on de la circunferencia unitaria que reposa en el primer cuadrante, y le llamaremos $\mathcal C$. En la primera parte de este trabajo construiremos una sucesi\'on de poligonales adecuadas para aproximar un arco $\arc{AB}$ de $\mathcal{C}$. Probaremos que la sucesi\'on de las longitudes de estas poligonales es convergente, y definiremos la longitud de $\arc{AB}$ como su l\'imite. Hemos dicho que las poligonales deben ser adecuadas porque de otro modo obtendr\'iamos resultados no deseados. Por ejemplo, si se construye una sucesi\'on con poligonales \emph{escalonadas} como la de la Figura \ref{fig0}, la sucesi\'on de longitudes ser\'ia constante e igual a la longitud de una poligonal formada por un \'unico escal\'on, lo cual obviamente echa por la borda el prop\'osito de aproximar el arco $\arc{AB}$.

\begin{figure}[htbp!]
\centering
\includegraphics[scale=0.7]{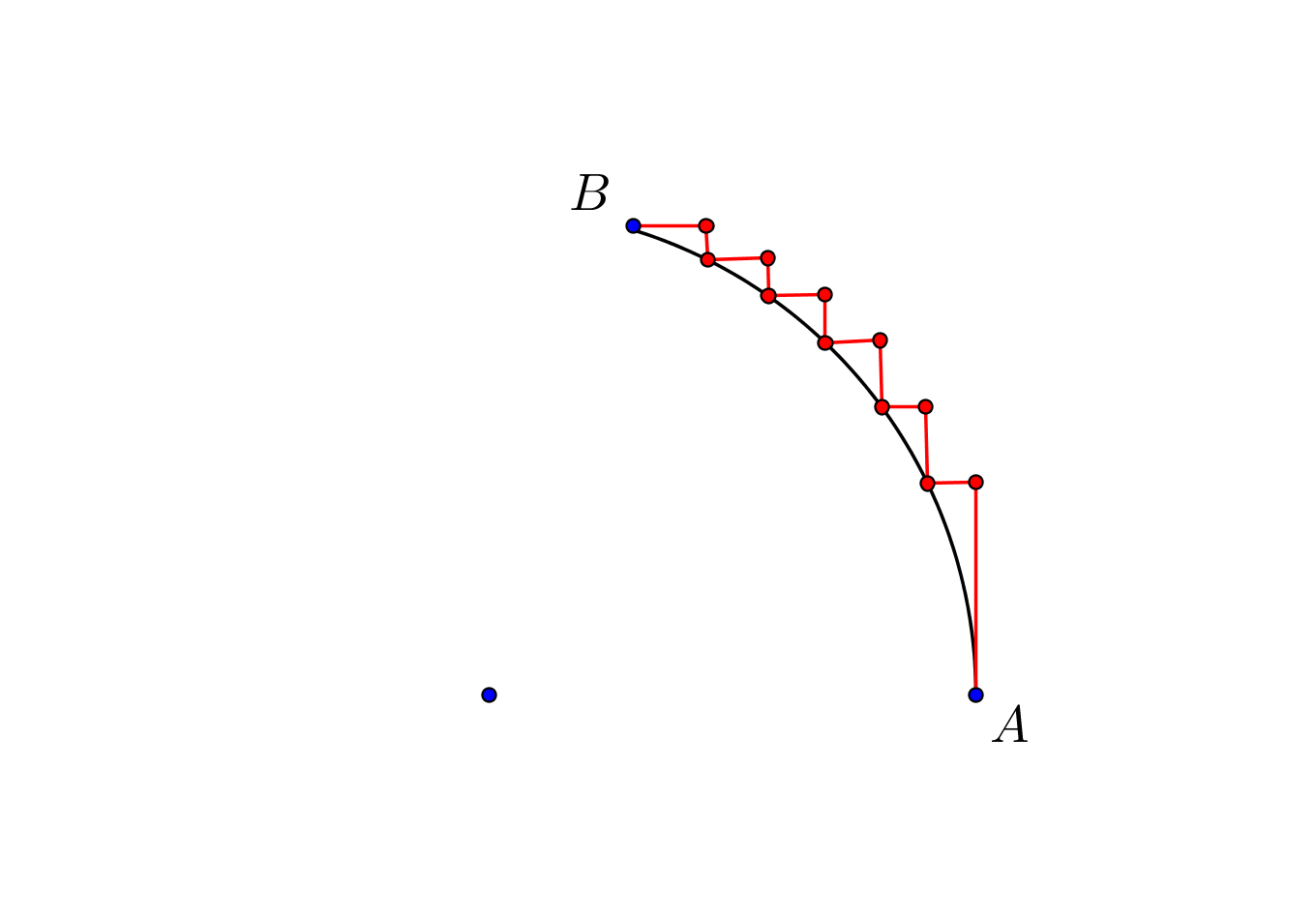}
\caption{\label{fig0} Una poligonal inapropiada.}
\end{figure}

\noindent\\
Adicionalmente, definiremos el \'area del sector circular determinado por $A$ y $B$, y mostraremos que esta guarda una relaci\'on de $1:2$ con la longitud de $\arc{AB}$. Esta relaci\'on ser\'a nuestra principal herramienta para responder la siguiente pregunta.\\
\\
2.\ ?`Es $\varphi$ una biyecci\'on entre $\mathcal{C}$ y $[0,\frac{\pi}{2}]$?\\
\\
Si queremos que la definici\'on de $\sen(x)$ que dimos arriba tenga sentido para cada $x\in [0,\frac{\pi}{2}]$, es necesario responder lo siguiente: ?`Es cierto que para cada $x$ existe un punto $P\in\mathcal{C}$ tal que $\varphi(P)=x$? ?`es \'unico este $P$? Si en lugar de un arco de circunferencia pensamos en la recta real, la primera propiedad se verifica de inmediato gracias al axioma del supremo, sin embargo no tenemos un ``axioma del supremo circular'' (!`tampoco queremos definir un nuevo axioma para cada curva que podamos imaginar!). Para resolver esto asignaremos, a cada $y$ tal que $0\le y\le 1$, la longitud del arco que va de $(1,0)$ al punto en $\mathcal{C}$ que tiene ordenada $y$ (esto corresponde, de hecho, a la definici\'on de $\arcsen(y)$). Probaremos que esta correspondencia es continua y mon\'otona, y aplicaremos el \textit{teorema del valor intermedio} (su enunciado y demostraci\'on pueden encontrarse, por ejemplo, en \cite[Teo. 5.2.1, p. 177]{Kit}) para concluir que es una biyecci\'on, lo que significar\'a que $\varphi$ tambi\'en lo es. Esto permitir\'a definir $\sen(x)$ para cada $x\in [0,\frac{\pi}{2}]$ sin ambigüedad como la ordenada del \'unico punto $P\in\mathcal C$ que satisface $\varphi(P)=x$.\\
\\
3.\ ?`Qu\'e sucede con la longitud de un arco de circunferencia si se usa otra sucesi\'on de poligonales para definirla?\\
\\
La respuesta a la primera pregunta entrega esencialmente un ejemplo de integraci\'on. Desde este punto de vista, y considerando que nuestra demostraci\'on de la continuidad de $\varphi$ depende de la relaci\'on de proporcionalidad probada en la primera parte, y por lo tanto tambi\'en de la definici\'on de la longitud de $\arc{AB}$, resulta vital preguntarse qu\'e ocurre si se elige una sucesi\'on de poligonales diferente para definir la longitud de $\arc{AB}$. En la tercera parte de este trabajo probaremos que si se elige cualquier sucesi\'on de poligonales que cumpla una determinada condici\'on, entonces el l\'imite de la sucesi\'on de longitudes asociada corresponde exactamente a la longitud definida en la primera parte.
\end{abstract}

\section{Longitud de arco y \'area de un sector circular}
Sea $\mathcal{C}$ el cuarto de circunferencia descrito en el resumen. Dados $A,B\in\mathcal{C}$ nos proponemos dos cosas: definir la longitud del arco $\arc{AB}$ y definir el \'area del sector circular determinado por este arco, que denotaremos por $S(A,B)$. Para construir nuestras definiciones a partir de objetos m\'as sencillos, como lo son los segmentos rectos y los pol\'igonos, describiremos el arco a partir de poligonales y el sector circular a partir de pol\'igonos. En la Figura \ref{fig1} se muestran una poligonal y un pol\'igono aproximando a $\arc{AB}$ y $S(A,B)$, respectivamente.

\begin{figure}[htbp!]
\centering
\includegraphics[scale=1]{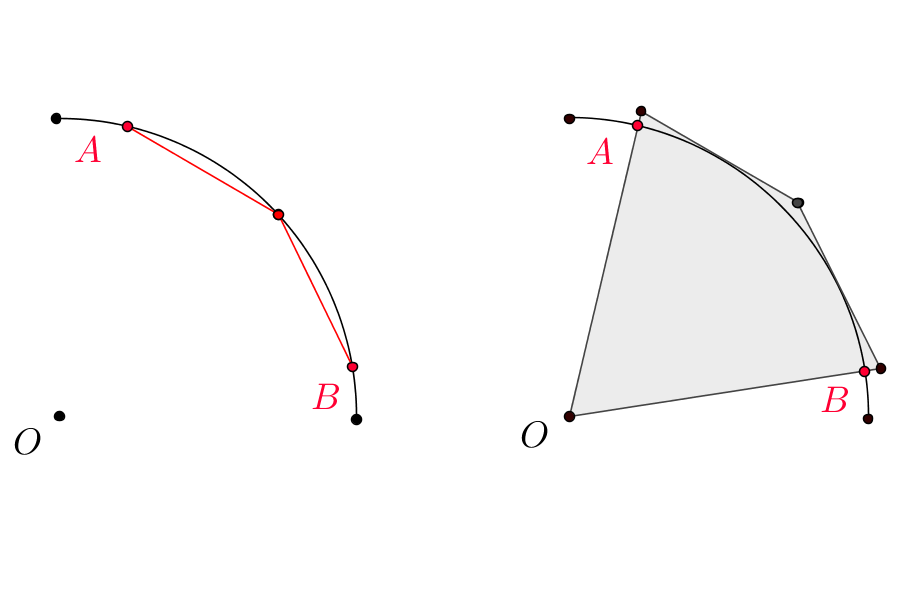}
\caption{\label{fig1} Aproximaci\'on por poligonales y pol\'igonos.}
\end{figure}

\subsection{Desarrollo de algunas herramientas}
Dado el rol que juega el teorema del valor intermedio (TVI) en este trabajo comenzaremos por enunciarlo.

\begin{theo}(\cite[Teo. 5.2.1, p. 177]{Kit}) Si $f:[a,b]\to\mathbb{R}$ es continua y $f(a)\not=f(b)$ entonces dado cualquier n\'umero $y$ entre $f(a)$ y $f(b)$ existe $x\in(a,b)$ tal que $f(x)=y$.
\end{theo}

\begin{lema}
Dados dos puntos $A$,$B$ en $\mathcal{C}$, existe un punto intermedio $P\in\mathcal{C}$ tal que $|AP|=|PB|\le\frac{|AB|}{\sqrt{2}}$. La circunferencia $\mathcal{C}$ queda por debajo de la recta que pasa por $P$ y es perpendicular a $OP$, donde $O=(0,0)$.
\end{lema}
\begin{proof}
Todo punto en $\mathcal{C}$ satisface su ecuaci\'on, a saber $x^2+y^2=1$. Esto permite escribir $A=(\sqrt{1-y_1^2},y_1),B=(\sqrt{1-y_2^2},y_2)$ para algunos $y_1,y_2\in (0,1)$ fijos.\footnote{La existencia de ra\'ices cuadradas puede justificarse aplicando el TVI a la funci\'on $f(x)=x^2$.} Para cada $y\in (0,1)$ consideremos el punto $P_y=(\sqrt{1-y^2},y)\in\mathcal{C}$ y definamos la funci\'on
$$h(y):=|AP_y|-|P_yB|$$
$$=\sqrt{\left(\sqrt{1-y^2}-\sqrt{1-y_1^2}\right)^2+(y-y_1)^2}-\sqrt{\left(\sqrt{1-y^2}-\sqrt{1-y_2^2}\right)^2+(y-y_2)^2}.$$
Esta funci\'on es continua en $(0,1)$ por obtenerse como suma, multiplicaci\'on y composici\'on de funciones continuas en los dominios correspondientes. Adem\'as $h(y_1)=0-|AB|=-(|AB|-0)=-h(y_2)$. Por el teorema del valor intermedio debe haber alg\'un $y$ entre $y_1$ e $y_2$ tal que 
$h(y)=0$; para este $y$ se verifica
$$|AP_y|-|P_yB|=0\Leftrightarrow |AP_y|=|P_yB|.$$

\noindent Para ver que $|AP|\le\frac{|AB|}{\sqrt{2}}$ consideremos la Figura \ref{fig2}, donde $P$ es tal que $|AP|=|PB|$. Aqu\'i $|OA|=|OB|$ (ambos segmentos son radios) y el trazo $OP$ es compartido por los tri\'angulos $\triangle OPA$ y $\triangle OPB$, de modo que ambos deben ser congruentes por el criterio lado-lado-lado, dado lo cual se tiene $\angle APQ=\angle BPQ$. Como el segmento $PQ$ es compartido por los tri\'angulos $\triangle QPA$ y $\triangle QPB$, por el criterio lado-\'angulo-lado podemos decir que estos \'ultimos son congruentes, luego $|AQ|=\frac{|AB|}{2}$ y $\angle AQP=\angle BQP$, por lo tanto
$$\pi=\angle AQP+\angle BQP=2\angle AQP\Rightarrow \angle AQP=\frac{\pi}{2}.$$

\begin{figure}[htbp!]
\centering
\includegraphics[scale=0.5]{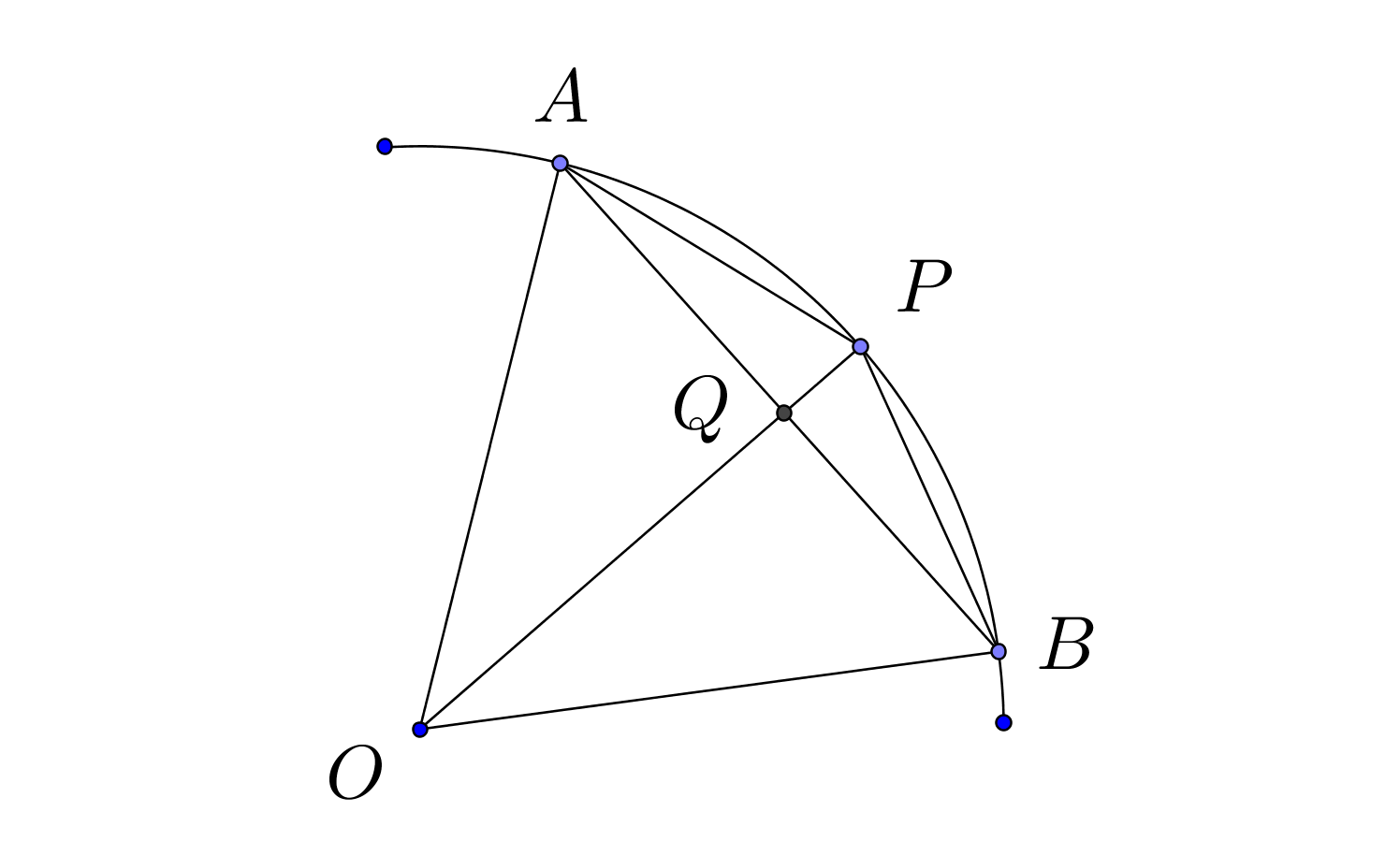}
\caption{\label{fig2} Punto ``medio'' entre $A$ y $B$.}
\end{figure}

\noindent As\'i, aplicando el teorema de Pit\'agoras al tri\'angulo $\triangle AQP$ y notando que $|OQ|+|QP|=1$ vemos que
$$|AP|^2=|AQ|^2+|QP|^2=|AQ|^2+(1-|OQ|)^2=|AQ|^2+|OQ|^2+1-2|OQ|.$$
Aplicando ahora el teorema de Pit\'agoras sobre $\triangle OQA$ y recordando que $OA$ es radio de la circunferencia, obtenemos
$$|AQ|^2+|OQ|^2+1-2|OQ|=1+1-2|OQ|=2(1-|OQ|)=\frac{2(1-|OQ|^2)}{1+|OQ|}$$
$$\le 2(1-|OQ|^2)=2|AQ|^2.$$
Como vimos antes, $|AQ|=\frac{|AB|}{2}$, por lo tanto hemos concluido que
$$|AP|^2\le 2|AQ|^2=2\cdot\left(\frac{|AB|}{2}\right)^2\Rightarrow |AP|\le \sqrt{2}\cdot\frac{|AB|}{2}=\frac{|AB|}{\sqrt{2}}.$$
Para probar la \'ultima afirmaci\'on del lema, sea $l$ una recta pasando por $P$, perpendicular a $OP$ (Figura \ref{fig3}). Esta recta divide el plano cartesiano en dos regiones $R_1$ y $R_2$ disjuntas como se muestra m\'as abajo, donde $l\subset R_1$. Si $X$ es cualquier punto en $R_2$, entonces hay dos opciones: $X$ est\'a en la recta obtenida al prolongar $OP$; $X$ no est\'a en tal recta. En el primer caso, vemos que $|OX|>|OP|=1$, luego $X\notin\mathcal{C}$. En el otro caso consideramos $OX$ y trazamos el segmento $YX$ paralelo a $l$, con $Y$ en la recta determinada por el segmento $OP$. Por el teorema de Pit\'agoras tendremos que
$$|OX|=\sqrt{|OY|^2+|YX|^2}> \sqrt{1+|YX|^2} >1,$$
lo que indica que $X\notin\mathcal{C}$. As\'i, si alg\'un punto $X'$ est\'a en $\mathcal{C}$ necesariamente debe tenerse que $X'\in R_1$.

\begin{figure}[htbp!]
\centering
\includegraphics[scale=0.55]{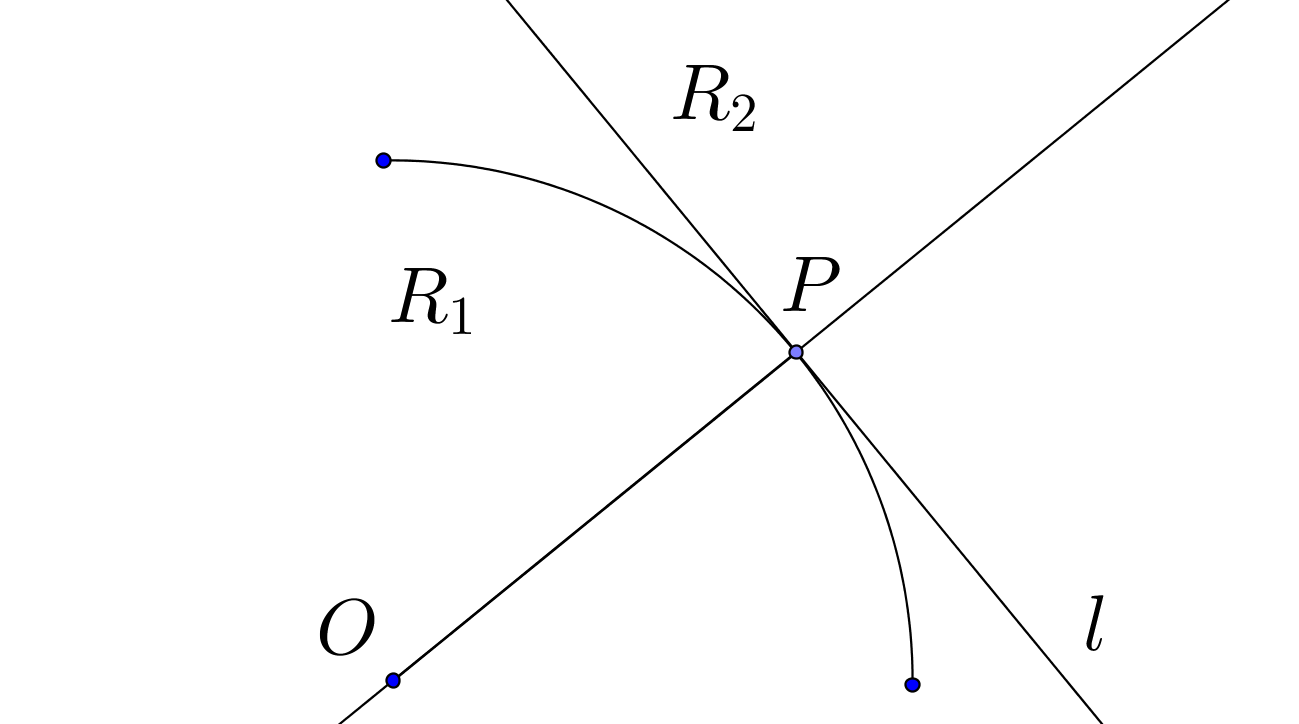}
\caption{\label{fig3}La circunferencia queda por debajo de la recta tangente.}
\end{figure}

\end{proof}

M\'as adelante, en el proceso de definir las magnitudes mencionadas al comienzo, obtendremos colecciones finitas de puntos $\{P_i\}_{i=0}^n$ en $\mathcal{C}$ y prestaremos atenci\'on a dos tipos de objetos asociados a cada una de ellas: una poligonal descrita por los puntos $A=P_0,P_1,\ldots,P_n=B$, y una colecci\'on de tri\'angulos $\triangle P_{i-1}OP_i$ disjuntos, como se muestra en la Figura \ref{fig4}.\\

\begin{figure}[htbp!]
\centering
\includegraphics[scale=0.43]{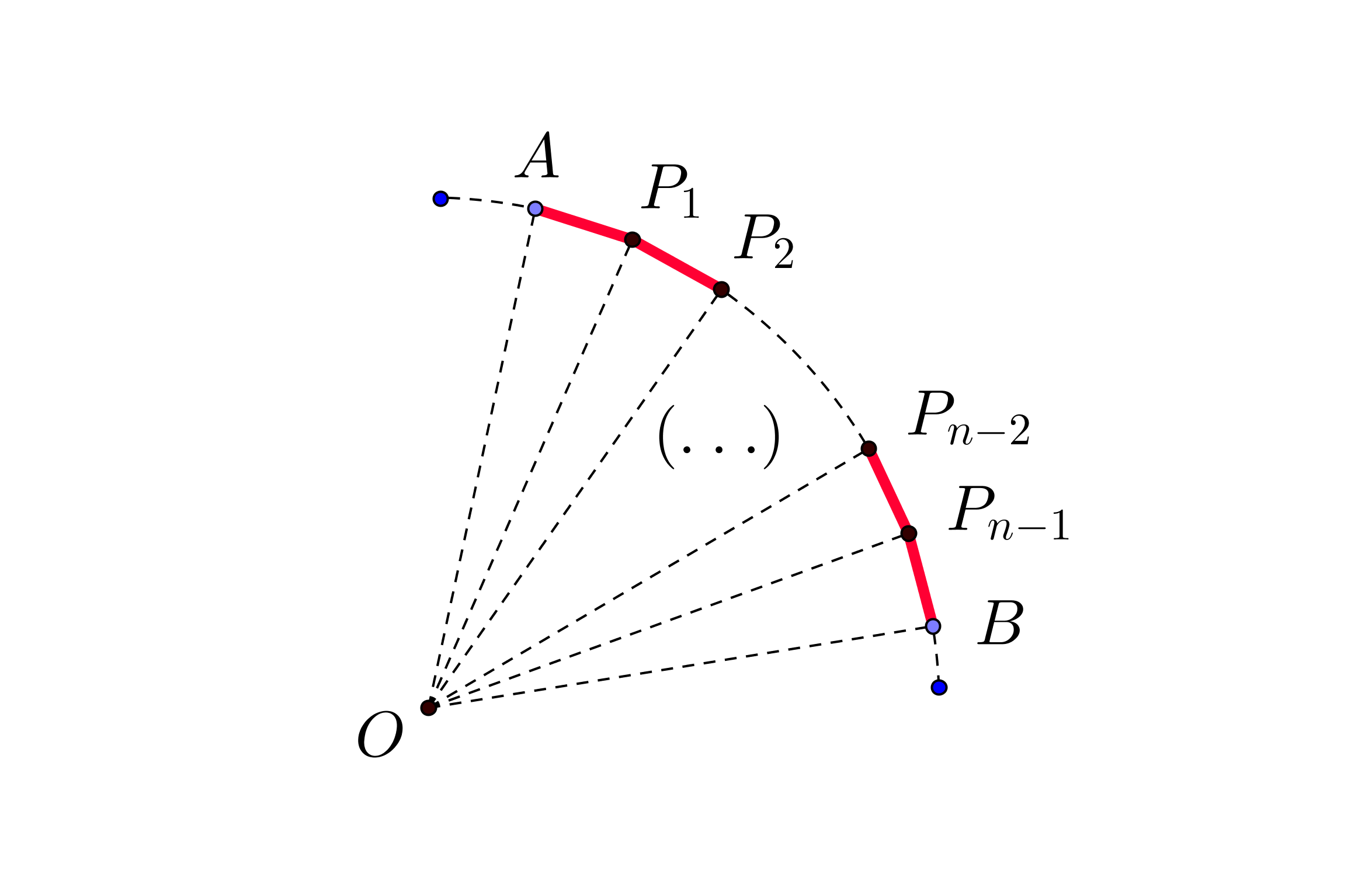}
\caption{\label{fig4}Una poligonal (en rojo) y los tri\'angulos correspondientes.}
\end{figure}

\noindent El siguiente lema permitir\'a probar que mientras m\'as fina sea la colecci\'on $\{P_i\}$ elegida mayor ser\'a el parecido entre las alturas en $O$ de los tri\'angulos $\triangle P_{i-1}OP_i$ y el radio de la circunferencia, y por lo tanto mayor ser\'a el parecido entre los pol\'igonos, uno interior y otro exterior, que se originen a partir de $\{P_i\}$. Como veremos en la sub-secci\'on 1.2, esta condici\'on ser\'a de vital importancia.

\begin{lema}
Dados $Q_1,Q_2$ sobre $\mathcal{C}$ entre $A$ y $B$, se tiene que $|Q_1Q_2|\le |AB|$.
\end{lema}

\begin{proof}
Dados $P_y,P_{y'}\in\mathcal{C}$, diremos que $P_y>P_{y'}$ si sus ordenadas $y,y'$ son tales que $y>y'$. Supongamos sin p\'erdida de generalidad que $A>Q_1>Q_2>B$. Escribamos entonces $A=(\sqrt{1-y_a^2},y_a)$, $Q_1=(\sqrt{1-y_1^2},y_1)$, $Q_2=(\sqrt{1-y_2^2},y_2)$, $B=(\sqrt{1-y_b^2},y_b)$, para $y_a>y_1>y_2>y_b$. Mostraremos que $|Q_1Q_2|\le |AB|$, viendo primero que $|Q_1Q_2|\le |AQ_2|$ y luego que $|AQ_2|\le |AB|$. Para lo primero calculemos $|Q_1Q_2|^2$ y probemos que $|Q_1Q_2|^2\le |AQ_2|^2$:
$$|Q_1Q_2|^2=\left(\sqrt{1-y_1^2}-\sqrt{1-y_2^2}\right)^2+\left(y_1-y_2\right)^2\le \left(\sqrt{1-y_1^2}-\sqrt{1-y_2^2}\right)^2+(y_a-y_2)^2.\ \ \ (*)$$
Por otra parte 
$$1>y_a>y_1>0\Rightarrow 0<\sqrt{1-y_a^2}<\sqrt{1-y_1^2}\Rightarrow 0>-\sqrt{1-y_a^2}>-\sqrt{1-y_1^2}$$
$$\Rightarrow \sqrt{1-y_2^2}-\sqrt{1-y_a^2}>\sqrt{1-y_2^2}-\sqrt{1-y_1^2}.$$
Adem\'as $\sqrt{1-y_2^2}-\sqrt{1-y_1^2}>0$ porque $y_2<y_1$, luego 
$$\left(\sqrt{1-y_2^2}-\sqrt{1-y_a^2}\right)^2\ge \left(\sqrt{1-y_2^2}-\sqrt{1-y_1^2}\right)^2.$$
De esto y $(*)$ se sigue que
$$|Q_1Q_2|^2\le  \left(\sqrt{1-y_2^2}-\sqrt{1-y_1^2}\right)^2+(y_a-y_2)^2$$
$$\le \left(\sqrt{1-y_2^2}-\sqrt{1-y_a^2}\right)^2+(y_a-y_2)^2=|AQ_2|^2,$$
dado lo cual resulta claro que $|Q_1Q_2|\le |AQ_2|$.\\
\\
Hemos probado en general que si $A>C>B$ son tres puntos sobre $\mathcal{C}$ entonces $|CB|\le |AB|$ y $|AC|\le |AB|$. Aplicando este resultado sobre $A$, $Q_2$ y $B$, podemos afirmar que $|AQ_2|\le |AB|$, lo que permite concluir lo requerido: $|Q_1Q_2|\le |AQ_2|\le |AB|$.
\end{proof}

\begin{figure}[htbp!]
\centering
\includegraphics[scale=0.5]{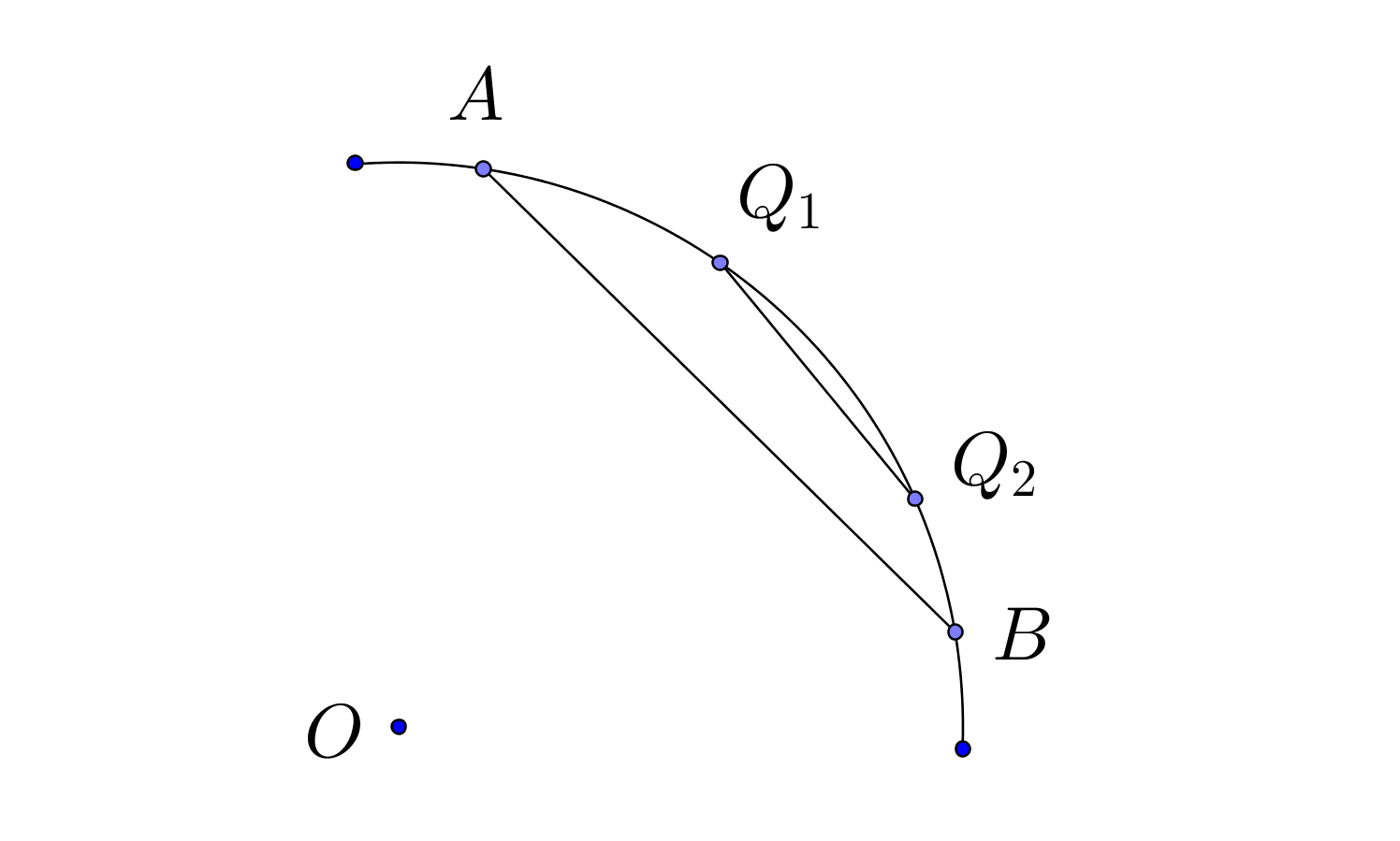}
\caption{\label{fig5} $|Q_1Q_2|\le|AB|$.}
\end{figure}

Una vez que hayamos definido la longitud de $\arc{AB}$ y el \'area de $S(A,B)$, el contenido de la siguiente proposici\'on permitir\'a, entre otras cosas, probar la relaci\'on de proporcionalidad que existe entre estas dos magnitudes.

\begin{prop}
\begin{itemize}
\item Sean $A=P_0>P_1>P_2>\cdots>P_n=B$ puntos en $\mathcal{C}$ (usando la noci\'on de orden fijada en la demostraci\'on del Lema 2). 
\item Sean $A'$ y $B'$ los puntos obtenidos al intersectar la recta tangente a $\mathcal{C}$ en $P$ (siendo $P$ el punto del Lema 1) con las rectas $OA$ y $OB$, respectivamente.
\item Para $i\in \{1,\ldots,n\}$ sea $T_i$ el tri\'angulo $\triangle P_{i-1}OP_i$, y sea $T_{ext}$ el tri\'angulo $\triangle A'OB'$.
\end{itemize}
Si $h_i$ es la altura en $O$ de $T_i$ y $\ell_i=|P_{i-1}P_i|$ para $1\le i\le n$, $\ell^{(0)}=|AB|$, y $h^{(0)}$ es la altura en $O$ de $\triangle AOB$, entonces
$$\ell^{(0)}\le\sum_{i=1}^n\ell_i\le\frac{2\text{\'area}(\bigcup T_i)}{\displaystyle\min_{1\le i\le n}h_i}\le\frac{2\text{\'area}(T_{ext})}{\displaystyle\min_{1\le i\le n}h_i}=\frac{\ell^{(0)}}{h^{(0)}\displaystyle\min_{1\le i\le n}h_i}.$$
\end{prop}

\begin{figure}[htbp!]
\centering
\includegraphics[scale=0.75]{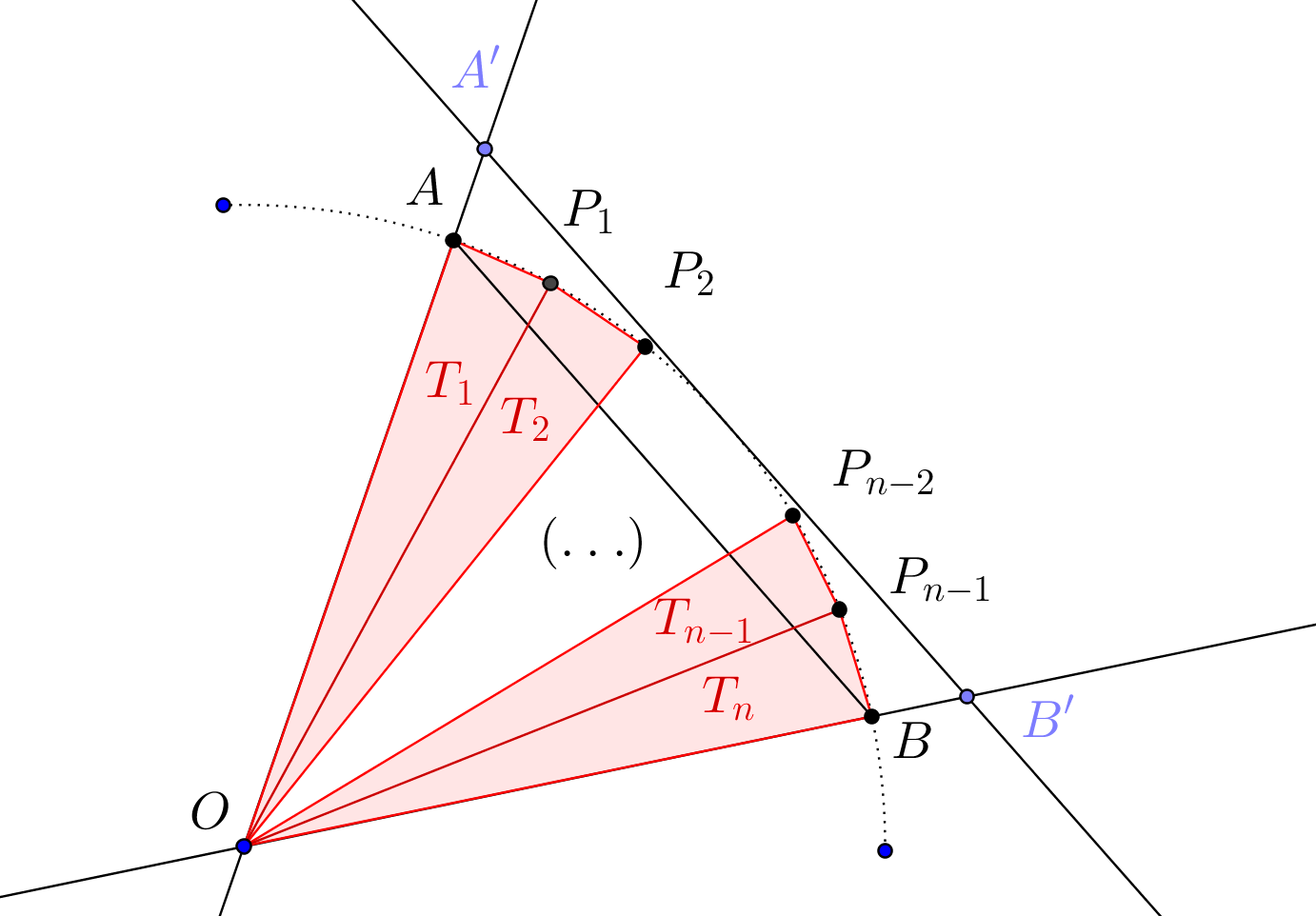}
\caption{\label{fig6} Elementos de la Proposici\'on 1.}
\end{figure}

\begin{proof}
\textbf{Primera desigualdad.} Cuando $n=2$ la desigualdad vale por desigualdad triangular:
$$\ell^{(0)}=|AB|=|AP_2|\le |AP_1|+|P_1P_2|=\ell_1+\ell_2.$$
Adoptemos ahora la hip\'otesis de que la desigualdad vale para alg\'un $n$ natural , es decir, supongamos que para cualquier colecci\'on de $n$ puntos $P_i\in\mathcal{C}$ tales que $A>P_1>\cdots>P_n=B$ se tiene $\ell^{(0)}\le\sum_{i=1}^n\ell_i$, y consideremos $n+1$ puntos en $\mathcal{C}$ entre $A$ y $B$; denot\'emoslos por $Q_1,Q_2,\ldots ,Q_{n+1}$. Consideremos uno de los puntos no extremos, digamos, el punto $Q_n$, ubicado entre $Q_{n-1}$ y $Q_{n+1}$. La desigualdad triangular asegura que $|Q_{n-1}Q_{n+1}|\le |Q_{n-1}Q_n|+|Q_nQ_{n+1}|$, con lo cual
$$\sum_{i=1}^{n+1}\ell_i=|AQ_1|+|Q_1Q_2|+\cdots +|Q_{n-1}Q_n|+|Q_nQ_{n+1}|$$
$$\ge |AQ_1|+|Q_1Q_2|+\cdots +|Q_{n-1}Q_{n+1}|.$$
Los $n$ puntos $Q_1,Q_2,\ldots ,Q_{n-1},Q_{n+1}$ est\'an en $\mathcal{C}$ entre $A$ y $B$, por lo tanto podemos usar nuestra hip\'otesis para afirmar que 
$$|AQ_1|+|Q_1Q_2|+\cdots +|Q_{n-1}Q_{n+1}|\ge \ell^{(0)},$$
lo que implica que $\sum_{i=1}^{n+1}\ell_i\ge \ell^{(0)}$. Por el principio de inducci\'on se sigue que la desigualdad vale para todo $n$ natural.\\
\\
\textbf{Segunda desigualdad.} Sea $T_i$ el tri\'angulo $\triangle OP_{i-1}P_i$, donde $P_0$ representa a $A$. Sean adem\'as $h_i$ la altura en $O$ del tri\'angulo $T_i$ y $\ell_i$ la longitud de la base $P_{i-1}P_i$. Tenemos entonces
$$\text{\'area}(T_i)=\frac{h_i\ell_i}{2}\Rightarrow 2\text{\'area}(T_i)=h_i\ell_i\ge \left(\min_{1\le j\le n}h_j\right)\cdot \ell_i$$
$$\Rightarrow \ell_i\le\frac{2\text{\'area}(T_i)}{\displaystyle\min_{1\le j\le n}h_j}\Rightarrow \sum_{i=1}^n\ell_i\le\sum_{i=1}^n\frac{2\text{\'area}(T_i)}{\displaystyle\min_{1\le j\le n}h_j}=\frac{2}{\displaystyle\min_{1\le j\le n}h_j}\sum_{i=1}^n\text{\'area}(T_i)=\frac{2}{\displaystyle\min_{1\le j\le n}h_j}\text{\'area}\left(\bigcup T_i\right),$$
donde la \'ultima igualdad se justifica con el hecho de que los $T_i$ son disjuntos.\\
\\
\textbf{Tercera desigualdad.} Todos los $T_i$ est\'an contenidos en el sector circular $S(A,B)$, el cual a su vez est\'a contenido en $T_{ext}$ seg\'un la \'ultima parte del Lema $1$. De aqu\'i necesariamente 
$$\text{\'area}\left(\bigcup T_i\right)\le \text{\'area}(T_{ext})\Rightarrow \frac{2\text{\'area}\left(\bigcup T_i\right)}{\displaystyle\min_{1\le i\le n}h_i}\le\frac{2\text{\'area}(T_{ext})}{\displaystyle\min_{1\le i\le n}h_i}.$$
\\
\textbf{La igualdad.} Para probar que $\displaystyle \frac{2\text{\'area}(T_{ext})}{\displaystyle\min_{1\le i\le n}h_i}= \frac{\ell^{(0)}}{h^{(0)}\displaystyle\min_{1\le i\le n}h_i}$ basta usar el teorema de Thales sobre los tri\'angulos $\triangle AOB$ y $\triangle A'OB'$:
\begin{align}\label{asdasd}\frac{|A'B'|}{\ell^{(0)}}=\frac{1}{h^{(0)}}\Rightarrow |A'B'|=\frac{\ell^{(0)}}{h^{(0)}}.\end{align}
Como $OP$ es radio y $A'B'$ es tangente a $\mathcal C$, ambos segmentos deben ser perpendiculares, gracias a lo cual podemos tomar $OP$ como altura de $T_{ext}$ y as\'i calcular su \'area usando la ecuaci\'on \eqref{asdasd} y el hecho de que $|OP|=1$:
$$\text{\'area}(T_{ext})=\frac{|A'B'|\cdot|OP|}{2}=\frac{|A'B'|}{2}=\frac{\ell^{(0)}}{2h^{(0)}}\Rightarrow \frac{2\text{\'area}(T_{ext})}{\displaystyle\min_{1\le i\le n}h_i}=\frac{2\cdot\frac{\ell^{(0)}}{2h^{(0)}}}{\displaystyle\min_{1\le i\le n}h_i}=\frac{\ell^{(0)}}{h^{(0)}\displaystyle\min_{1\le i\le n}h_i}.$$
\end{proof}

\subsection{Definici\'on de $\stackrel{\textstyle\frown}{\mathrm{AB}}$ a partir de  poligonales}

Para definir con precisi\'on la longitud del arco $\arc{AB}$,
necesitamos tomar una decisi\'on respecto de la naturaleza que queremos asignarle  (tal como tendremos que hacerlo m\'as adelante al definir el \'area de $S(A,B)$).
Siguiendo nuestra intuici\'on, describiremos $\arc{AB}$
a partir de  una sucesi\'on de poligonales. Este procedimiento entrega la materia prima que usaremos tambi\'en en la secci\'on siguiente para construir pol\'igonos adecuados al prop\'osito de definir el \'area del sector circular $S(A,B)$.\\
\\
Dados $A,B\in\mathcal{C}$, sea $\ell^{(0)}:=|AB|$ y sea $h^{(0)}$ la altura en $O$ del tri\'angulo $\triangle AOB$ (como en la Proposici\'on 1). Sea $P$ un punto en $\mathcal{C}$ tal que $|AP|=|PB|\le\frac{|AB|}{\sqrt{2}}$ (afirmamos que $P$ existe y verifica esta desigualdad gracias al Lema 1). La poligonal formada al dibujar los trazos $\overline{AP}$ y $\overline{PB}$ consiste de dos segmentos de longitud $\ell^{(1)}:=|AP|\le\frac{|AB|}{\sqrt{2}}$, y tiene una longitud total igual a $L_1:=2\ell^{(1)}$. Tomando ahora puntos $R,S\in\mathcal{C}$ tales que $|AR|=|RP|$ y $|PS|=|SB|$, y dibujando nuevamente trazos rectos entre puntos consecutivos, obtenemos una nueva poligonal formada por cuatro segmentos de longitud $\ell^{(2)}:=|AR|\le\frac{|AP|}{\sqrt{2}}\le\frac{|AB|}{\sqrt{2}^2}$, y longitud total igual a $L_2:=4\ell^{(2)}$. Procediendo inductivamente, en la $m$-\'esima iteraci\'on de este proceso contamos con una colecci\'on de puntos $\{A=P_0,P_1,P_2,\ldots,P_{2^m}=B\}$, y con ello una poligonal formada por $2^m$ segmentos de longitud $\ell^{(m)}$ verificando $\ell^{(m)}\le\frac{|AB|}{\sqrt{2}^m}$, y longitud total igual a $L_m:=2^m\ell^{(m)}$. Obtenemos adem\'as $2^m$ tri\'angulos $T_i=\triangle P_{i-1}OP_i$ congruentes y disjuntos entre s\'i; llamemos $h^{(m)}$ a la altura en $O$ de cualquiera de ellos.\\
\\
Naturalmente, nos interesamos en las poligonales formadas porque constituyen una v\'ia razonable de acercamiento a la naturaleza del arco $\arc{AB}$: decidimos que la longitud de arco de $\arc{AB}$, de ser definida, debe corresponder a una magnitud \'intimamente relacionada a la sucesi\'on de longitudes $\{L_m\}_{m\in\mathbb{N}}$ descrita arriba. Cabe notar adem\'as que estamos tomando como punto de partida la siguiente afirmaci\'on: ``la longitud de una poligonal es igual a la suma de las longitudes de los segmentos rectos que la forman''.

\begin{prop}
La sucesi\'on de longitudes $\{L_m\}$ ya descrita es creciente y acotada superiormente.
\end{prop}

\begin{proof}
Sea $m\in\mathbb N$ arbitrario. En la $m$-\'esima iteraci\'on tenemos $2^m$ segmentos de longitud $\ell^{(m)}$ y la misma cantidad de tri\'angulos congruentes, con altura en $O$ igual a $h^{(m)}$. En t\'erminos de la Proposici\'on 1, podemos observar que $\ell_1=\ell_2=\ldots=\ell_{2^m}=\ell^{(m)}$ y que $h_1=h_2=\ldots=h_{2^m}=h^{(m)}$. Esta \'ultima nos dice que 
$$\sum_{i=1}^{2^m}\ell_i\le\frac{\ell^{(0)}}{h^{(0)}\displaystyle\min_{1\le i\le 2^m}h_i},$$
lo cual implica, por la observaci\'on que acabamos de hacer, que

\begin{align}\label{1}2^m\ell^{(m)}\le\frac{\ell^{(0)}}{h^{(0)}h^{(m)}}.\end{align}

Por una parte, el miembro de la izquierda corresponde a $L_m$, y por otra, el de la derecha es menor que $\displaystyle\frac{\ell^{(0)}}{\left(h^{(0)}\right)^2}$. En efecto, el Lema 2 asegura que $\ell^{(m)}\le \ell^{(0)}$, de modo que $\displaystyle\sqrt{1-\left(\frac{\ell^{(0)}}{2}\right)^2}\le\sqrt{1-\left(\frac{\ell^{(m)}}{2}\right)^2}$. Ahora, al aplicar el Teorema de Pit\'agoras sobre los tri\'angulos de la Figura \ref{fig7} nos damos cuenta de que la desigualdad anterior dice precisamente que $h^{(0)}\le h^{(m)}$, lo que confirma lo que afirmamos arriba, a saber,
$$\frac{\ell^{(0)}}{h^{(0)}h^{(m)}}\le\frac{\ell^{(0)}}{\left(h^{(0)}\right)^2}.$$
\begin{figure}[htbp!]
\centering
\includegraphics[scale=0.9]{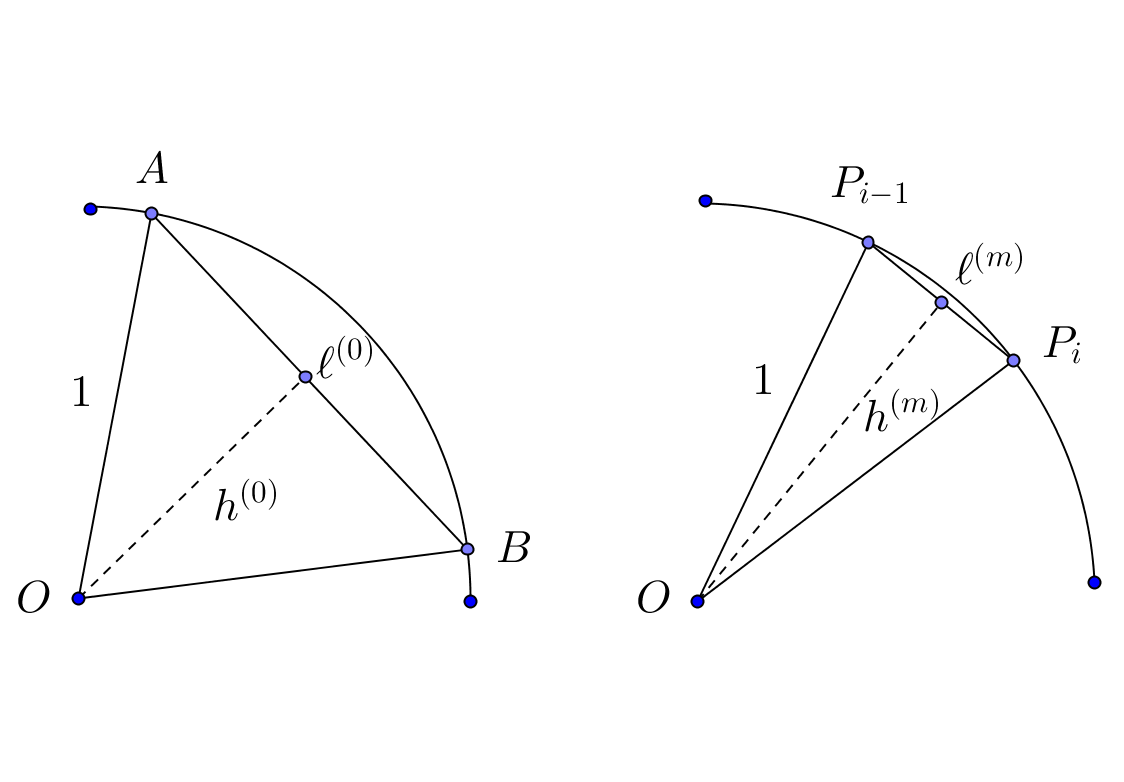}
\caption{\label{fig7} $\displaystyle\left(h^{(0)}\right)^2+\left(\frac{\ell^{(0)}}{2}\right)^2=1
\ \ \text{y}\ \ \left(h^{(m)}\right)^2+\left(\frac{\ell^{(m)}}{2}\right)^2=1.$}
\end{figure}

\noindent Al recopilar todo esto en la relaci\'on $\eqref{1}$ concluimos finalmente que
\begin{align}\label{cota}L_m\le\frac{\ell^{(0)}}{\left(h^{(0)}\right)^2},\ \ \ \forall m\in\mathbb{N},\end{align}

\noindent lo que prueba que $\displaystyle\frac{\ell^{(0)}}{\left(h^{(0)}\right)^2}$ es una cota superior de $\{L_m\}$.\\
\\
El hecho de que $\{L_m\}$ es creciente se sigue de la Proposici\'on 1. Tomando $\ell^{(m)}$ en el lugar de $\ell^{(0)}$, y observando que la poligonal $m+1$--\'esima se obtiene tomando un punto medio entre cada par de puntos consecutivos de la poligonal $m$-\'esima, podemos notar lo siguiente: sobre cada segmento de longitud $\ell^{(m)}$ del paso $m$ aparecen dos segmentos de longitud $\ell^{(m+1)}$ en el paso $m+1$. As\'i, la primera desigualdad de la Proposici\'on 1 dice que
$$\ell^{(m)}\le\sum_{i=1}^2\ell^{(m+1)},$$
de donde se desprende (sumando a ambos lados) que
$$L_m=\sum_{j=1}^{2^m}\left[\ell^{(m)}\right]\le\sum_{j=1}^{2^m}\left[\sum_{i=1}^2\ell^{(m+1)}\right]=\sum_{j=1}^{2^m}2\ell^{(m+1)}$$
$$= 2^m\cdot2\ell^{(m+1)}=2^{m+1}\ell^{(m+1)}=L_{m+1}.$$
\end{proof}

En virtud del axioma del supremo, se sigue de la proposici\'on anterior que la sucesi\'on $\{L_m\}$ es convergente, dado lo cual definimos la longitud del arco $\arc{AB}$ como
$$|\arc{AB}|:=\lim_{m\to\infty}L_m.$$

\subsection{Definici\'on del \'area de  $S(A,B)$ a partir de  poligonales}

Nos enfrentamos ahora a la misi\'on de definir el \'area de $S(A,B)$. 
Antes de continuar aclararemos que, en adelante, 
usaremos la notaci\'on $|\cdot |$ para representar tanto \'areas como longitudes, debiendo el lector realizar la interpretaci\'on correcta dependiendo del contexto. Por ejemplo, al hablar de $|\triangle OPQ|$ nos referimos al \'area del tri\'angulo con v\'ertices $O$, $P$ y $Q$, mientras que por $|PQ|$ nos referimos a la longitud del segmento $PQ$.\\
\\
Aunque el objetivo de definir el \'area de una figura parece intuitivamente sencillo, definir rigurosamente la \textit{medida} de un subconjunto de $\mathbb{R}^2$ es una tarea que requiere de un trabajo minucioso, y corresponde a un objetivo de la \textit{teor\'ia de la medida}. Como no es nuestro prop\'osito entrar en los detalles que esta teor\'ia aborda (el lector interesado puede consultar, por ejemplo, \cite[cap\'itulos 1 y 2]{Hen}), s\'olo mencionaremos que hay ciertos subconjuntos de $\mathbb{R}^2$ que no se pueden medir (no se les puede asignar un \'area), y por lo tanto necesitamos criterios que nos permitan decidir cu\'ando un conjunto dado es medible. De acuerdo a esto, tomaremos las siguientes premisas como punto de partida:

\begin{itemize}
\item Los pol\'igonos en $\mathbb{R}^2$ son medibles, y el \'area de cualquier pol\'igono es un real no negativo; asumiremos que el \'area de un tri\'angulo cualquiera viene dada por la cl\'asica expresi\'on $\frac{\text{base}\cdot\text{altura}}{2}$, y que $\left|\bigcup_{k=1}^nS_k\right|=\sum_{k=1}^n|S_k|$ cuando $\{S_k\}_{k=1}^n$ es una familia de pol\'igonos disjuntos en $\mathbb{R}^2$.
\item Si un conjunto $S\subset\mathbb{R}^2$ es tal que para cualquier $\varepsilon>0$ existen dos pol\'igonos $\Sigma$ y $\Sigma'$ en $\mathbb{R}^2$ verificando $$\Sigma\subset S\subset \Sigma'\text{\ \ \ \ y\ \ \ \ \ }|\Sigma'|-|\Sigma|<\varepsilon,$$
entonces $S$ es medible y su \'area puede definirse como el \'infimo de las \'areas de los pol\'igonos que lo contienen.
\end{itemize}

\noindent De acuerdo a esto, antes de definir el \'area de $S(A,B)$ necesitamos probar que esta regi\'on es medible. Ahora bien, probar esto \'ultimo en t\'erminos de las premisas de arriba se traduce en tomar $S=S(A,B)$ y luego hallar dos pol\'igonos $\Sigma$ y $\Sigma'$ que verifiquen las hip\'otesis de la segunda premisa. Para lograr esto usaremos el procedimiento descrito al comienzo de esta subsecci\'on de la siguiente manera.\\
\\
En la $m$-\'esima iteraci\'on de nuestro procedimiento de obtenci\'on de puntos y poligonales, definamos $\Sigma_m$ como el pol\'igono de v\'ertices $OP_0P_1P_2\ldots P_{2^m}$. Este est\'a contenido en $S(A,B)$ y consiste en la uni\'on disjunta de los tri\'angulos $T_i$ de la Proposici\'on 1 para $1\le i\le 2^m$, es decir, $\displaystyle\Sigma_m=\bigcup_{i=1}^{2^m}T_i$. Adem\'as se tiene que $|\Sigma_m|=\displaystyle\sum_{i=1}^{2^m}|T_i|$.\\
\\
Ahora (usaremos notaci\'on nueva con el \'unico fin de definir un pol\'igono exterior; olvidarse de los siguientes elementos despu\'es no supondr\'a problemas), para cada $i\in \{1,\ldots,2^m\}$ sea $M_i$ un punto en $\mathcal{C}$ tal que $|P_{i-1}M_i|=|M_iP_i|$, y sea $\gamma_i$ la recta tangente a $\mathcal{C}$ en $M_i$. Sean ahora $P_{i-1}'$ y $P_i'$ los puntos de intersecci\'on de $\gamma_i$ con las prolongaciones de las rectas $\overline{OP_{i-1}}$ y $\overline{OP_i}$ respectivamente.\\

\begin{figure}[htbp!]
\centering
\includegraphics[scale=0.65]{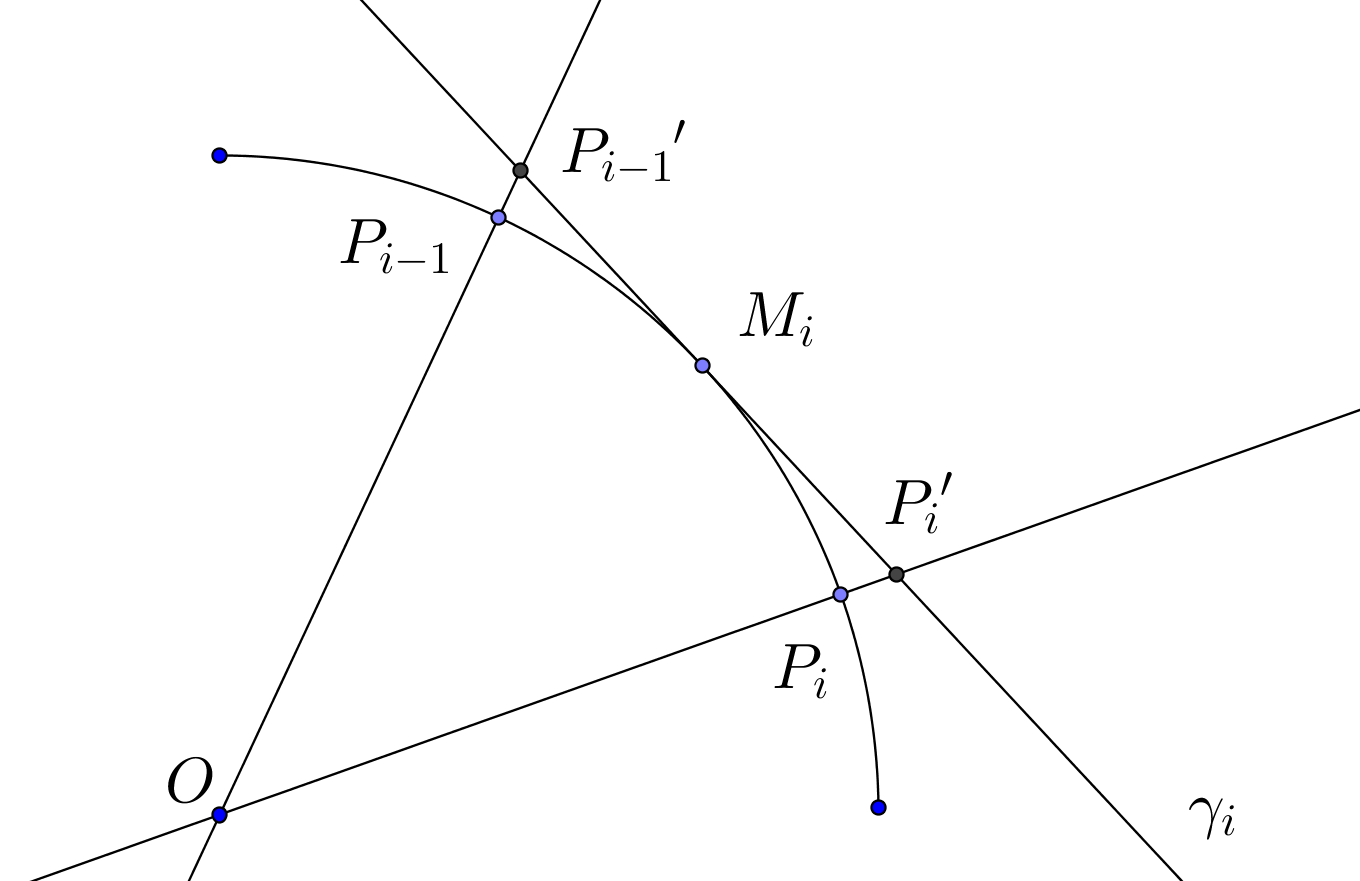}
\caption{\label{fig8} V\'ertices para un pol\'igono exterior.}
\end{figure}

\noindent Si escribimos $\triangle P_{i-1}'OP_i'=T_i'$, entonces definimos $\Sigma_m':=\bigcup T_i'$. De manera an\'aloga a $\Sigma_m$, el pol\'igono $\Sigma_m'$ contiene a $S(A,B)$ (el Lema 1 asegura que el arco con extremos $P_{i-1}$ y $P_i$ queda por debajo del segmento $\overline{P_{i-1}'P_i'}$) y adem\'as por construcci\'on $|\Sigma_m'|=\displaystyle\sum_{i=1}^{2^m}|T_i'|$.\\
\\
Lo primero que observamos es que $\Sigma_m\subset S(A,B)\subset\Sigma_m'$ para todo $m\in\mathbb{N}$, es decir, se cumple una de las hip\'otesis de la segunda premisa. Ahora solamente falta hallar alg\'un $m_0$ tal que $|\Sigma_{m_0}'|-|\Sigma_{m_0}|<\varepsilon$ para cualquier $\varepsilon>0$ dado con anterioridad.

\begin{prop}
Sean $\ell^{(m)},\ell^{(0)}$, y $h^{(0)}$ las magnitudes descritas al comienzo de esta subsecci\'on. Dado $\varepsilon>0$ existe $m_0\in\mathbb{N}$ tal que si $m\ge m_0$ entonces
\begin{align}
\label{alturas}\frac{1}{1-\left(\frac{\ell^{(m)}}{2}\right)^2}<1+\frac{2\varepsilon \left(h^{(0)}\right)^2}{\ell^{(0)}}.
\end{align}
\end{prop}

Nota: El rol del t\'ermino $\frac{2\varepsilon \left(h^{(0)}\right)^2}{\ell^{(0)}}$ se reduce a representar una magnitud arbitrariamente peque\~na; no escribimos $\varepsilon$ por motivos casi est\'eticos que veremos despu\'es de la demostraci\'on.

\begin{proof}
Gracias a las primeras reflexiones de esta subsecci\'on, sabemos que $\ell^{(m)}\le\frac{|AB|}{\sqrt{2}^m}$ para cada $m\in\mathbb{N}$. Como $\frac{|AB|}{\sqrt{2}^m}\overset{m\to\infty}{\longrightarrow}0$, se sigue del teorema del sandwich que $\ell^{(m)}\overset{m\to\infty}{\longrightarrow}0$, luego $\displaystyle \frac{1}{1-\left(\frac{\ell^{(m)}}{2}\right)^2}\overset{m\to\infty}{\longrightarrow}1$ gracias a la aritm\'etica de l\'imites. Por definici\'on de convergencia, dado que $\frac{2\varepsilon \left(h^{(0)}\right)^2}{\ell^{(0)}}$ es positivo, debe existir $m_0\in\mathbb{N}$ tal que 
$$\displaystyle 1-\frac{2\varepsilon \left(h^{(0)}\right)^2}{\ell^{(0)}}<\frac{1}{1-\left(\frac{\ell^{(m)}}{2}\right)^2}<1+\frac{2\varepsilon \left(h^{(0)}\right)^2}{\ell^{(0)}},\ \ \ \forall m\ge m_0.$$
\end{proof}

Sea $\varepsilon>0$ arbitrario y sea $m_0$ el natural de la proposici\'on anterior. Afirmamos que $|\Sigma_m'|-|\Sigma_m|<\varepsilon$ siempre que $m\ge m_0$. En efecto, al calcular $|T_i|$ tomando la altura en $O$ y $|T_i'|$ como al final de la demostraci\'on de la Proposici\'on 1, vemos que
$$|\Sigma_m'|-|\Sigma_m|=\sum_{i=1}^{2^m}|T_i'|-\sum_{i=1}^{2^m}|T_i|=\sum_{i=1}^{2^m}\frac{\ell^{(m)}}{2h^{(m)}}-\sum_{i=1}^{2^m}\frac{\ell^{(m)}h^{(m)}}{2}=\frac{2^{m}\ell^{(m)}}{2h^{(m)}}-\frac{2^m\ell^{(m)}h^{(m)}}{2}$$
$$=\frac{1}{2}\cdot \underbrace{2^m\ell^{(m)}}_{=L_m}h^{(m)}\left[\frac{1}{\left(h^{(m)}\right)^2}-1\right]=\frac{1}{2}\cdot L_mh^{(m)}\left[\frac{1}{1-\left(\frac{\ell^{(m)}}{2}\right)^2}-1\right]$$
$$\underbrace{<}_{\text{desigualdad \eqref{alturas}}}\frac{L_mh^{(m)}}{2}\left[1+\frac{2\varepsilon \left(h^{(0)}\right)^2}{\ell^{(0)}}-1\right].$$
Como vimos en la desigualdad $\eqref{cota}$, $L_m\le\frac{\ell^{(0)}}{\left(h^{(0)}\right)^2}$, y adem\'as $h^{(m)}=\frac{1}{\sqrt{1-\left(\frac{\ell^{(m)}}{2}\right)^2}}<1$, por lo tanto 
$$\frac{L_mh^{(m)}}{2}\left[1+\frac{2\varepsilon \left(h^{(0)}\right)^2}{\ell^{(0)}}-1\right]<\frac{1}{2}\cdot\frac{\ell^{(0)}}{\left(h^{(0)}\right)^2}\cdot 1\left[\frac{2\varepsilon \left(h^{(0)}\right)^2}{\ell^{(0)}}\right]=\varepsilon.$$
Todo esto prueba que basta elegir cualquier $m$ mayor que $m_0$ para asegurar que $|\Sigma_m'|-|\Sigma_m|<\varepsilon$, de modo que existe un par de pol\'igonos verificando lo que quer\'iamos (de hecho, existen infinitos de ellos). La conclusi\'on consiguiente es que el sector $S(A,B)$ es medible, lo que nos da la licencia de definir su \'area como el \'infimo de las \'areas de los pol\'igonos que lo contienen:
$$|S(A,B)|:=\inf_{\Sigma}\{|\Sigma|:\Sigma\text{ es un pol\'igono que contiene a }S(A,B)\}.$$

Para concluir esta secci\'on probaremos la relaci\'on que nos permitir\'a demostrar la continuidad de la funci\'on $\arcsen$.

\begin{theo}
Dados $A,B\in\mathcal{C}$, se tiene que
$$|\arc{AB}|=2|S(A,B)|.$$
\end{theo}

\begin{proof}
La Proposici\'on 1 nos dice que
$\displaystyle\sum_{i=1}^n\ell_i\le\frac{2\left|\bigcup T_i\right|}{\displaystyle\min_{1\le i\le n}h_i}\le\frac{2\left|T_{ext}\right|}{\displaystyle\min_{1\le i\le n}h_i}=\frac{\ell^{(0)}}{h^{(0)}\displaystyle\min_{1\le i\le n}h_i}$. Sea $m_0$ el natural usado anteriormente y sea $m>m_0$ un natural cualquiera. Si consideramos las iteraciones $m_0$ y $m$ de nuestro procedimiento, tendremos (entre otras cosas) dos colecciones de tri\'angulos, una compuesta por $2^m$ tri\'angulos interiores que dan forma a $\Sigma_m$, y otra compuesta por $2^{m_0}$ tri\'angulos $T_j'$ exteriores que dan forma a $\Sigma_{m_0}'$. Llamemos $T_{ext}$ a cualquiera de los tri\'angulos de esta \'ultima colecci\'on, y llamemos $T_i$ al $i$-\'esimo de los $2^{m-m_0}$ tri\'angulos de la primera colecci\'on que est\'an contenidos en $T_{ext}$.\\

\begin{figure}[htbp!]
\centering
\includegraphics[scale=0.7]{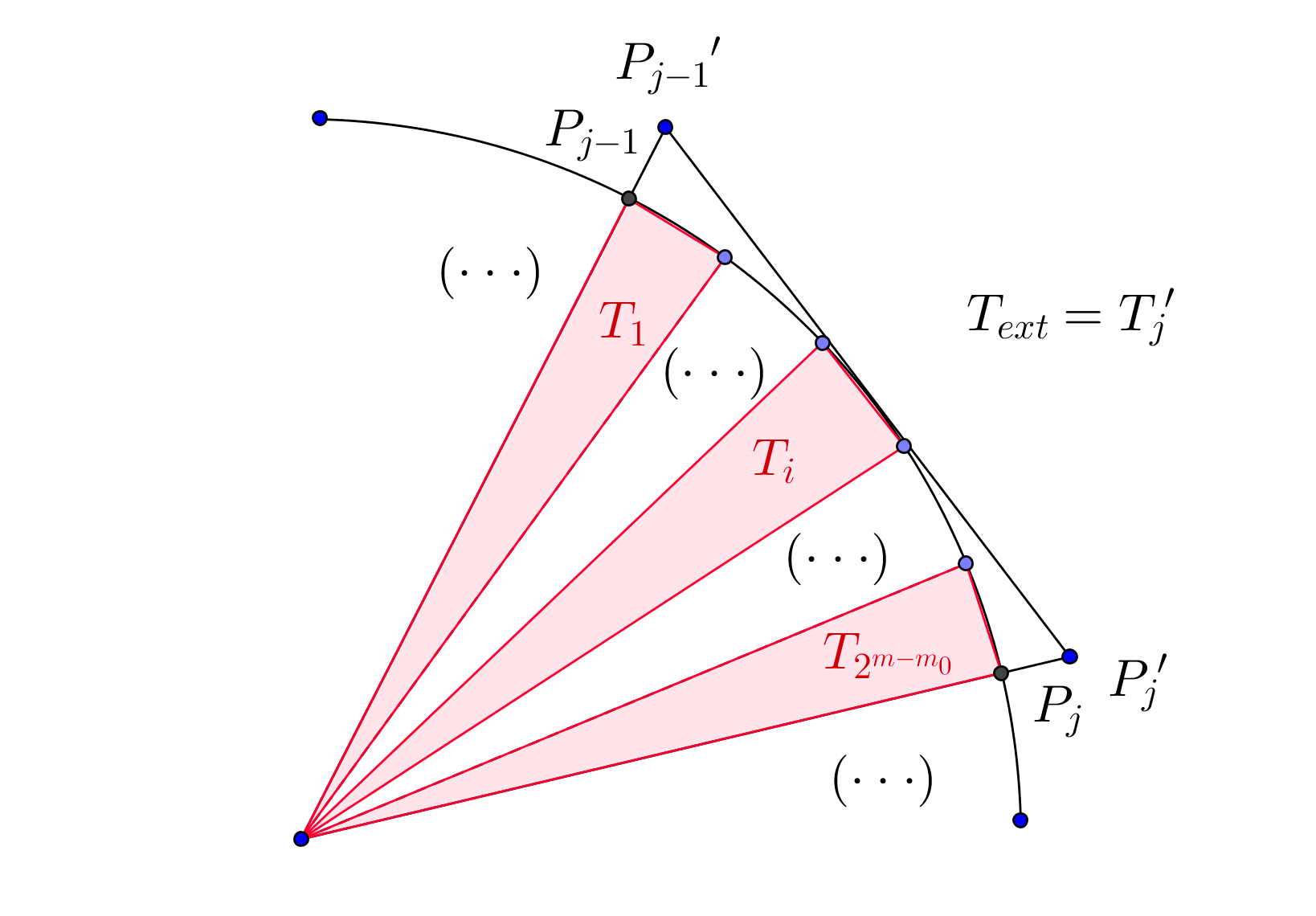}
\caption{\label{fig9} Tri\'angulos interiores $T_i$ contenidos en $T_{ext}$.}
\end{figure}

Como ya explicamos en la demostraci\'on de la Proposici\'on 2, tenemos $\displaystyle h_1=h_2=\ldots=h_{2^{m-m_0}}=h^{(m)}=\min_{1\le i\le 2^{m-m_0}}h_i$, y $\ell_i=\ell^{(m)}$, por lo tanto la cadena de desigualdades anterior, aplicada a $n=2^{m-m_0}$ y tomando $\ell^{(m_0)}$ y $h^{(m_0)}$ en lugar de $\ell^{(0)}$ y $h^{(0)}$, puede reescribirse como sigue:
$$2^{m-m_0}\ell^{(m)}\le \frac{ 2|\bigcup T_i|}{h^{(m)}}\le\frac{2|T_{ext}|}{h^{(m)}}=\frac{\ell^{(m_0)}}{h^{(m_0)}h^{(m)}}.$$
Sumando estas desigualdades a lo largo a todos los $2^{m_0}$ tri\'angulos $T_j'$ (y notando que todas las cantidades involucradas no dependen del \'indice del tri\'angulo $T_j'$), obtenemos
$$2^{m_0}\cdot 2^{m-m_0}\ell^{(m)}\le 2^{m_0}\cdot \frac{2\left|\bigcup T_i\right|}{h^{(m)}}\le 2^{m_0}\cdot\frac{2|T_{ext}|}{h^{(m)}}=\frac{2^{m_0}\ell^{(m_0)}}{h^{(m_0)}h^{(m)}}$$
$$\qquad\Leftrightarrow L_m\le 2\cdot\frac{|\Sigma_m|}{h^{(m)}}\le 2\cdot\frac{|\Sigma_{m_0}'|}{h^{(m)}}=\frac{L_{m_0}}{h^{(m_0)}h^{(m)}}.$$
De aqu\'i se desprenden dos cosas:
\begin{itemize}
\item Como la Proposici\'on 3 asegura que $\displaystyle\frac{1}{h^{(m)}}=\frac{1}{\sqrt{1-\left(\frac{\ell^{(m)}}{2}\right)^2}}<\sqrt{1+\frac{2\varepsilon \left(h^{(0}\right)^2}{\ell^{(0)}}}$, se sigue de la primera desigualdad que

\begin{align}\label{una}L_m\le 2|\Sigma_m|\sqrt{1+\frac{2\varepsilon \left(h^{(0)}\right)^2}{\ell^{(0)}}}.\end{align}

\item La igualdad dice que $\displaystyle 2|\Sigma_{m_0}'|=\frac{L_{m_0}}{h^{(m_0)}}$, y sabemos adem\'as que $\displaystyle\frac{1}{h^{(m_0)}}<\sqrt{1+\frac{2\varepsilon \left(h^{(0}\right)^2}{\ell^{(0)}}}$, por lo tanto

\begin{align}
\label{dos} 2|\Sigma_{m_0}'|\le L_{m_0}\sqrt{1+\frac{2\varepsilon \left(h^{(0)}\right)^2}{\ell^{(0)}}}.
\end{align}
El hecho de que $\Sigma_m\subset S(A,B)$ garantiza que para cualquier pol\'igono $\Sigma'$ conteniendo a $S(A,B)$ se tiene $|\Sigma_m|\le|\Sigma'|$, de modo que 
$$|\Sigma_m|\le\inf\left\lbrace|\Sigma'|:\Sigma'\text{ es pol\'igono que contiene a }S(A,B)\right\rbrace=|S(A,B)|,$$
lo cual puede aplicarse en la relaci\'on $\eqref{una}$ para establecer que

\begin{align}\label{conc1}L_m\le 2|S(A,B)|\sqrt{1+\frac{2\varepsilon \left(h^{(0)}\right)^2}{\ell^{(0)}}}.\end{align}

Por otra parte, $|S(A,B)|=\inf\left\lbrace|\Sigma'|:\Sigma'\text{ es pol\'igono que contiene a }S(A,B)\right\rbrace$ es menor o igual que $|\Sigma_{m_0}'|$ porque $\Sigma_{m_0}'$ contiene a $S(A,B)$. Aplicando esto en la relaci\'on $\eqref{dos}$ vemos que
\begin{align}\label{conc2}2|S(A,B)|\le L_{m_0}\sqrt{1+\frac{2\varepsilon \left(h^{(0)}\right)^2}{\ell^{(0)}}}.\end{align}

Uniendo lo obtenido en $\eqref{conc1}$ y $\eqref{conc2}$ podemos escribir lo siguiente:

$$L_m\le 2|S(A,B)|\sqrt{1+\frac{2\varepsilon \left(h^{(0)}\right)^2}{\ell^{(0)}}}\le L_{m_0}\left(1+\frac{2\varepsilon \left(h^{(0)}\right)^2}{\ell^{(0)}}\right).
$$

Recordando que la sucesi\'on $\{L_m\}$ es creciente y $\displaystyle\lim_{m\to\infty}L_m=|\arc{AB}|$, resulta claro que $\displaystyle L_{m_0}\left(1+\frac{2\varepsilon \left(h^{(0)}\right)^2}{\ell^{(0)}}\right)\le |\arc{AB}|\left(1+\frac{2\varepsilon \left(h^{(0)}\right)^2}{\ell^{(0)}}\right)$, lo que permite concluir que
$$L_m\le 2|S(A,B)|\sqrt{1+\frac{2\varepsilon \left(h^{(0)}\right)^2}{\ell^{(0)}}}\le |\arc{AB}|\left(1+\frac{2\varepsilon \left(h^{(0)}\right)^2}{\ell^{(0)}}\right).$$
Como esto vale para todo $m\ge m_0$, las desigualdades deben mantenerse en el l\'imite cuando $m\to\infty$, es decir,
$$|\arc{AB}|\le 2|S(A,B)|\sqrt{1+\frac{2\varepsilon \left(h^{(0)}\right)^2}{\ell^{(0)}}}\le |\arc{AB}|\left(1+\frac{2\varepsilon \left(h^{(0)}\right)^2}{\ell^{(0)}}\right).$$
Hemos llegado a una relaci\'on que s\'olo depende de $\varepsilon$ (recordemos que $A$ y $B$ son fijos), y que es v\'alida para cada $\varepsilon>0$, de modo que las desigualdades perduran cuando $\varepsilon\to 0^+$. Esto \'ultimo nos lleva a concluir que $|\arc{AB}|\le 2|S(A,B)|\le |\arc{AB}|$, lo cual s\'olo es posible si 
$$|\arc{AB}|=2|S(A,B)|.$$

\end{itemize}

\end{proof}

\section{Existencia de $\sen(x)$ para $x\in\left[0,\frac{\pi}{2}\right]$ arbitrario}
Con el resultado del teorema anterior bajo la manga nos enfocaremos ahora en lo que nos motiv\'o a realizar este trabajo.
\begin{defi}
Sea $\mathcal{C}$ el cuarto de circunferencia usado en la secci\'on anterior, y sean $Q=(1,0)$ e $Y=(\sqrt{1-y^2},y)$. Para cada $y\in [0,1]$ definimos
$$\arcsen(y):=\left\lbrace\begin{array}{l}|\arc{YQ}|,$ si $y>0\\
0,$\ \ \ \ \ si $y=0\end{array}\right. .$$
\end{defi}

Esta definici\'on, a diferencia de la definici\'on que se da para $\sen(x)$, responde de manera inmediata a la interrogante que planteamos en el resumen de este trabajo: sin importar el real $y$ entre $0$ y $1$ elegido, el valor $\arcsen(y)$ estar\'a siempre bien definido.\\
\\
Nuestro objetivo consiste en mostrar que esta funci\'on es continua, y que por lo tanto satisface la propiedad del valor intermedio; esto \'ultimo, sumado a la monoton\'ia de $\arcsen$ (que demostraremos m\'as adelante), permitir\'a garantizar la existencia de $\sen(x)$ bajo su definici\'on usual para cualquier valor de $x\in\left[0,\frac{\pi}{2}\right]$. En concreto, queremos probar que para cualquier $y_0\in [0,1]$ se cumple que $\displaystyle\lim_{y\to y_0}|\arcsen(y)-\arcsen(y_0)|=0$.\\
\\
Dado que se trata de una funci\'on no algebraica, lograr el v\'inculo entre las magnitudes $|y-y_0|$ y $|\arcsen(y)-\arcsen(y_0)|$ resulta ser una tarea delicada. El teorema mostrado al final de la secci\'on anterior juega un papel fundamental en este punto, ya que permite vincular $\arcsen(y)$ con una magnitud m\'as d\'ocil al momento de probar continuidad. Sin hablar m\'as, definamos $g(y):=|S(Y,Q)|$, y notemos que $\arcsen(y)=2g(y)$ para cada $y\in [0,1]$, de modo que $\arcsen$ es una funci\'on continua si y s\'olo si $g$ lo es.

\begin{prop}
La funci\'on $g:[0,1]\to\mathbb{R}$ definida arriba es continua.
\end{prop}

\begin{proof}
Sea $y_0\in [0,1]$ arbitrario y sea $Y_0$ el punto en $\mathcal{C}$ de coordenadas $(x_0,y_0)=(\sqrt{1-y_0^2},y_0)$. Para cada $y\not=y_0$, $y\in [0,1]$, sea $Y=(x,y)=(\sqrt{1-y^2},y)$ y sea $Z=\gamma_1\cap\gamma_2$, siendo $\gamma_1$ la recta tangente a $\mathcal{C}$ en $Y_0$ y $\gamma_2$ la recta que pasa por $\overline{OY}$, como se muestra en la Figura \ref{fig10}.\\

\begin{figure}[htbp!]
\centering
\includegraphics[scale=0.48]{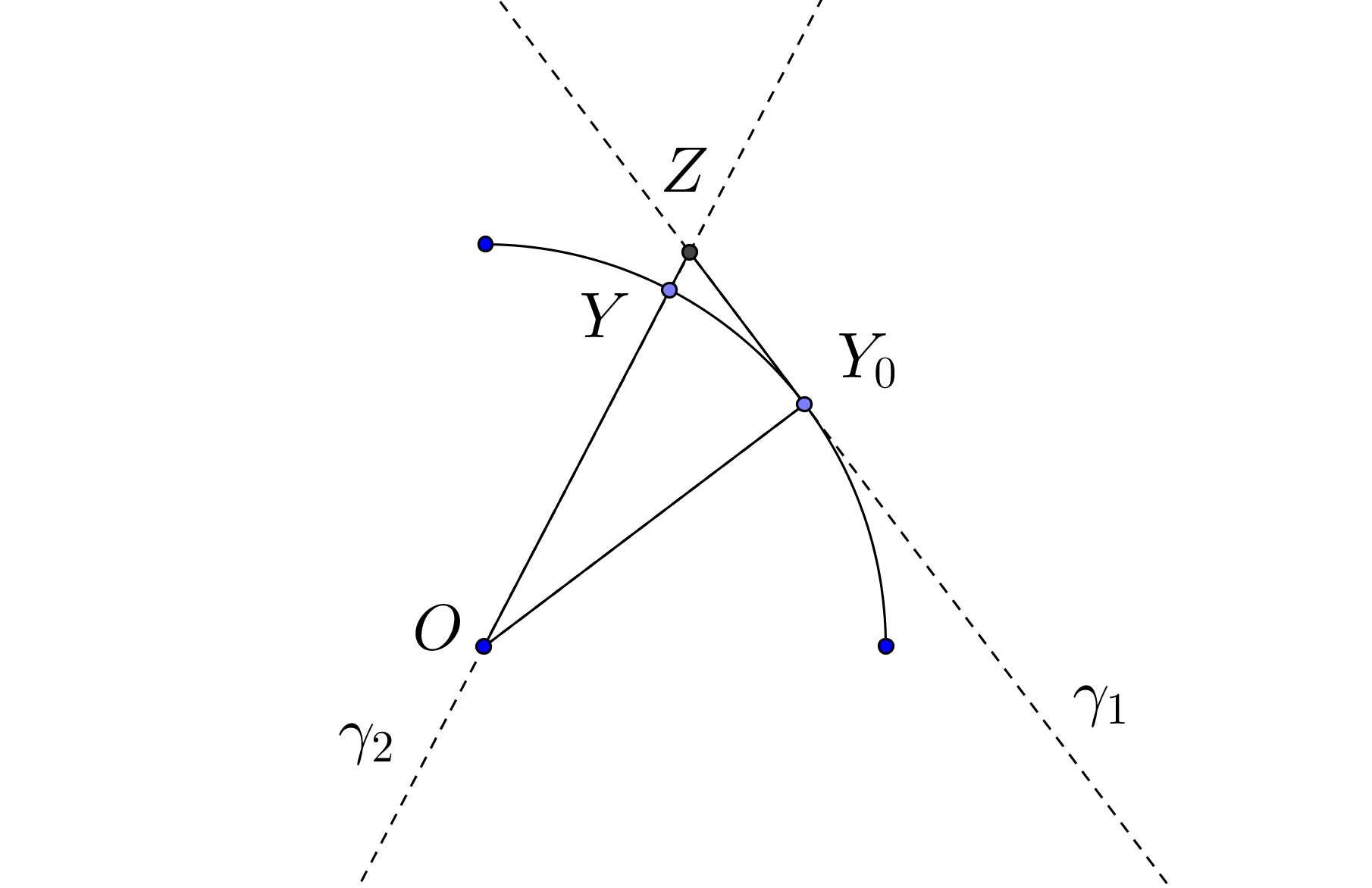}
\caption{\label{fig10} Continuidad de $g$.}
\end{figure}

\noindent El \'area del sector circular $S(Y,Y_0)$, que viene dada por $|g(y)-g(y_0)|$ (al final de la secci\'on 3 daremos una justificaci\'on para la aditividad del \'area de sectores circulares), es menor o igual que el \'area del tri\'angulo $\triangle ZOY_0$ para todo $y$ (consecuencia de la definici\'on de \'area que dimos y del Lema 1), de modo que intentaremos probar que $|\triangle ZOY_0|\to 0$ cuando $y\to y_0$ para concluir que $|g(y)-g(y_0)|\to 0$ cuando $y\to y_0$.\\
\\
El hecho de que $|\triangle ZOY_0|$ tiende a cero se sigue del hecho de que $|\overline{ZY_0}|$ tiende a cero cuando $y\to y_0$. En efecto, si denotamos por $(u,v)$ al vector $\overrightarrow{Y_0Z}$, como $\overrightarrow{Y_0Z}\perp \overrightarrow{OY_0}$ el producto punto entre ambos vectores debe ser igual a cero, es decir,
\begin{align}
\label{ppunto}ux_0+vy_0=0.
\end{align}
Por otra parte, si denotamos por $Y'$ y $Z'$ a las proyecciones de los puntos $Y$ y $Z$ sobre el eje $X$, respectivamente, podemos aplicar el Teorema de Thales sobre los tri\'angulos $\triangle YOY'$ y $\triangle ZOZ'$ (Figura \ref{fig11}) para obtener

\begin{align}
\label{thales}\frac{x_0+u}{x}=\frac{y_0+v}{y}.
\end{align}

\begin{figure}[htbp!]
\centering
\includegraphics[scale=0.5]{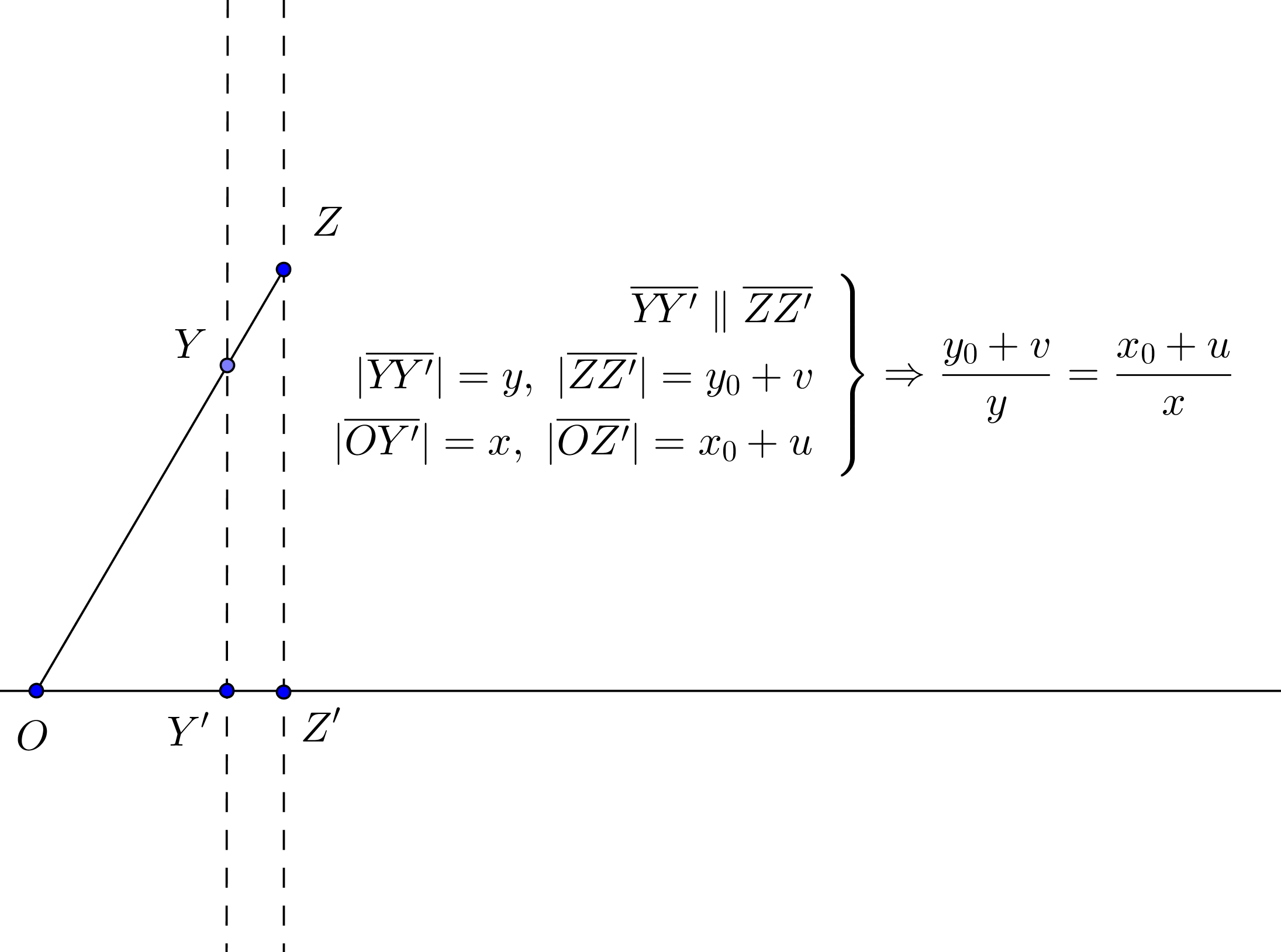}
\caption{\label{fig11} Teorema de Thales sobre $\triangle YOY$ y $\triangle ZOZ'$.}
\end{figure}

\noindent Todo esto permite hallar expl\'icitamente $\overrightarrow{Y_0Z}$, cuyas coordenadas son la soluci\'on $(u,v)$ al sistema formado por las ecuaciones $\eqref{ppunto}$ y $\eqref{thales}$:
$$\overrightarrow{Y_0Z}=(u,v)=\left(\frac{y_0}{xx_0+yy_0}(xy_0-x_0y),\frac{-x_0}{xx_0+yy_0}(xy_0-x_0y)\right).$$
Como $x=\sqrt{1-y^2}\overset{y\to y_0}{\longrightarrow}\sqrt{1-y_0^2}=x_0$, el resultado es el que quer\'iamos:
$$(u,v)\overset{y\to y_0}{\longrightarrow}\left(\frac{y_0}{x_0^2+y_0^2}(x_0y_0-x_0y_0),\frac{-x_0}{x_0^2+y_0^2}(x_0y_0-x_0y_0)\right)=(0,0).$$

\noindent Como dijimos en un comienzo, se verifica que $0\le |g(y)-g(y_0)|\le |\triangle ZOY_0|$, lo que permite obtener la continuidad de $g$:
$$|\triangle ZOY_0|\overset{y\to y_0}{\longrightarrow}0\Rightarrow |g(y)-g(y_0)|\overset{y\to y_0}{\longrightarrow}0.$$
\end{proof}

Observaci\'on: Si bien la Figura \ref{fig10} ilustra el caso $y>y_0$, los argumentos son completamente an\'alogos (y v\'alidos) para $y<y_0$.\\
\\
Se sigue directamente de la relaci\'on $\arcsen(y)=2g(y)$ y de las propiedades aritm\'eticas de los l\'imites que la funci\'on $\arcsen$ debe ser continua. Como dijimos en un comienzo, esto y el teorema del valor intermedio garantizan que para cada $x\in \left[0,\frac{\pi}{2}\right]$\footnote{En un sentido estricto, a\'un no sabemos que la longitud de arco $x$ pertenezca al intervalo $\left[0,\frac{\pi}{2}\right]$ ya que, hasta ahora, s\'olo podemos describir la longitud de $\mathcal{C}$ como el l\'imite (apenas sabemos que existe) de una sucesi\'on. Este asunto pasa por entender qu\'e representa $\pi$.  Hist\'oricamente, $\pi$ se ha definido (en realidad, de manera equivalente a lo que describiremos) como la longitud de un \textit{arco extendido} en una circunferencia unitaria; en un paso intelectual propio, podemos aplicar el procedimiento de la subsecci\'on 1.2 a la semicircunferencia que descansa en los cuadrantes $I$ y $II$, en lugar de aplicarlo s\'olo a $\mathcal{C}$, y luego definir $\pi$ como el l\'imite de la sucesi\'on de longitudes de las poligonales correpondientes. Como mostraremos al final del documento, las longitudes de arcos de circunferencia son aditivas, lo cual permite decir, junto con nuestra definici\'on de $\pi$, que la longitud del arco que va de $(1,0)$ a $(0,1)$ es exactamente igual a la mitad de la longitud del arco extendido, esto es, $\frac{\pi}{2}$.} existe al menos un $y\in [0,1]$ tal que $\arcsen(y)=x$. Sin embargo esto no es suficiente; si queremos definir $\sen(x)$ como el \'unico valor de $y\in [0,1]$ tal que $\arcsen(y)=x$, debemos probar precisamente aquello que acabamos de decir: tal valor de $y$ es \'unico.\\
\\
Esto no es dif\'icil de probar si observamos que, dados $a,b\in [0,1]$, se tiene que $\arcsen(a)\not=\arcsen(b)$ si y s\'olo si $g(a)\not=g(b)$, de modo que basta verificar la inyectividad de $g$ para concluir la de $\arcsen$. Afortunadamente, la funci\'on $g$ es estrictamente creciente gracias a nuestras premisas acerca de la medida de pol\'igonos.\\
\\
En efecto, sean $a,b\in [0,1]$ tales que $a>b$, y sea $\Sigma'$ un pol\'igono cualquiera conteniendo al sector $S(P(a),Q)$, para $Q=(1,0)$. Trazando una recta por $\overline{OP(b)}$ notamos que $\Sigma'$ puede escribirse como la uni\'on disjunta de dos pol\'igonos $\Sigma'_1$ y $\Sigma_2'$, donde $S_1:=S(P(a),P(b))\subset \Sigma_1'$ y $S_2:=S(P(b),Q)\subset \Sigma_2'$. En la Figura \ref{fig11} mostramos un posible pol\'igono $\Sigma'$ y los dos pol\'igonos $\Sigma_1'$ y $\Sigma_2'$ asociados.

\begin{figure}[htbp!]
\centering
\includegraphics[scale=0.45]{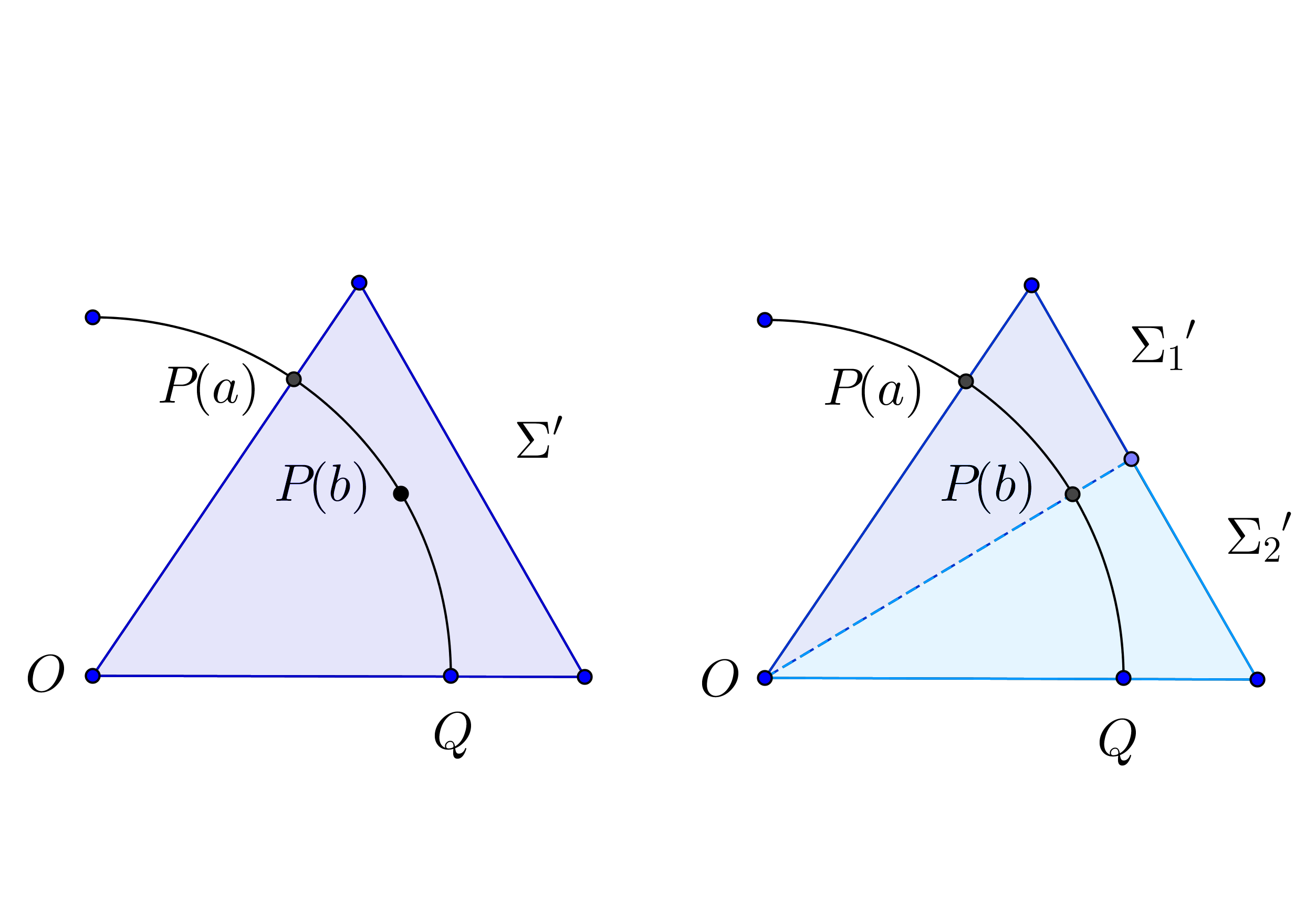}
\caption{\label{fig12} Un pol\'igono exterior arbitrario.}
\end{figure}

\noindent La definici\'on que dimos del \'area de un sector circular nos asegura que
$$|S_1|+|S_2|\le |\Sigma_1'|+|\Sigma_2'|=|\Sigma'|.$$

\noindent Como $\Sigma'$ es arbitrario, la desigualdad anterior se mantiene al tomar el \'infimo sobre todos los posibles valores de $|\Sigma'|$, de modo que

$$|S_1|+|S_2|\le |S(P(a),Q)|=g(a).$$ 
Notando que $|S_1|+|S_2|>|S_2|=g(b)$ ($\triangle P(a)OP(b)$ est\'a contenido en cualquier pol\'igono que contiene a $S_1$ y $|P(a)OP(b)|>0$, luego $S_1=\inf\{|\Sigma|:S_1\subset \Sigma,\Sigma\text{ pol\'igono}\}>0$)  concluimos que $g(b)<g(a)$, lo que implica la inyectividad de $g$, y por lo tanto tambi\'en la de $\arcsen$. Todo esto permite decir que, dado $x\in \left[0,\frac{\pi}{2}\right]$, existe un \'unico $y\in [0,1]$ tal que $\arcsen(y)=x$, al cual podemos denotar por $\sen(x)$.

\section{Unicidad de la definici\'on de longitud de arco}

Para definir la longitud del arco $\arc{AB}$ hemos apostado nuestras fichas a la idea de que las poligonales que construimos se parecen (en un sentido intuitivo) m\'as al arco a medida que el \'indice $m$ aumenta. Razonablemente, alguien podr\'ia encontrar una sucesi\'on distinta de poligonales que tambi\'en cumpla la caracter\'istica anterior (por ejemplo, dividiendo el arco $\arc{AB}$ en tres arcos en lugar de dos, y luego repitiendo este proceso sobre cada uno de los arcos obtenidos), y luego definir $|\arc{AB}|$ en t\'erminos de ella, lo que nos lleva de inmediato a la siguiente duda: ?`Ser\'a posible que $|\arc{AB}|$ tenga dos valores diferentes al definirse mediante dos sucesiones de poligonales distintas, a pesar de que ambas cumplan con parecerse m\'as y m\'as a $\arc{AB}$? Nos gustar\'ia poder decir que la respuesta a esta pregunta es ``no, la longitud de $\arc{AB}$ no depende de la sucesi\'on de poligonales elegida'', pero para ello necesitamos aclarar algo: ?`Qu\'e significa en concreto que una sucesi\'on est\'e formada por poligonales que se parecen m\'as y m\'as al arco $\arc{AB}$? Decidiremos que una poligonal se parece a $\arc{AB}$ cuando los segmentos de ella se ``confunden'' con el arco, es decir, cuando las longitudes de todos ellos son suficientemente peque\~nas (Lema 2 y Proposici\'on 3). La siguiente definici\'on nos permitir\'a decir esto mismo con menos palabras.

\begin{defi}
Sea $\mathcal{P}=\{P_0,P_1,P_2,\ldots,P_n\}$ una colecci\'on finita ordenada de puntos (puntos $P_i$ cuyas ordenadas $y_i$ verifican $y_{i-1}>y_i,\ \forall i$) en $\arc{AB}$ tales que $P_0=A$ y $P_n=B$. Diremos que esta colecci\'on es una partici\'on de $\arc{AB}$ y definimos su norma por
$$\|\mathcal{P}\|:=\max_{i\in \{1,\ldots,n\}}|P_{i-1}P_i|.$$
\end{defi}

\begin{defi}
Dada $\mathcal{P}=\{P_0,P_1,\ldots,P_n\}$ una partici\'on de $\arc{AB}$, diremos que una poligonal formada por $n$ segmentos est\'a asociada a $\mathcal{P}$ si los v\'ertices de ella son puntos de $\mathcal{P}$ y si adem\'as el $i$-\'esimo segmento est\'a dado por $\overline{P_{i-1}P_i}$.
\end{defi}

Observaci\'on: Bajo estas definiciones, toda poligonal asociada a alguna partici\'on de $\mathcal{C}$ tiene sus v\'ertices en orden decreciente respecto de sus ordenadas.\\
\\
En virtud de la discusi\'on anterior a las definiciones, consideraremos solamente sucesiones de poligonales que est\'en asociadas a sucesiones de particiones cuya norma tienda a cero. Es decir, si una sucesi\'on de particiones $\{\mathcal{P}_m\}$ cumple que
\begin{align}
\label{norma}\lim_{m\to\infty}\|\mathcal{P}_m\|=0,
\end{align}
entonces reci\'en admitiremos a la sucesi\'on de poligonales asociadas como un candidato para definir $|\arc{AB}|$. El motivo de exigir esto viene de que, si $\delta>0$ es arbitrario y $\|\mathcal P_m\|<\delta,\ \forall m\ge m_0$ para alg\'un $m_0$, entonces $|P_{i-1}P_i|<\delta,\ \forall i\in \{1,\ldots ,m\}$ para cada $m\ge m_0$; esto no habla de otra cosa que el parecido que buscamos entre las poligonales y $\arc{AB}$.

\begin{defi}
Sea $\mathcal{P}$ una partici\'on de $\arc{AB}$. Si $\mathcal Q$ es un conjunto finito de puntos en $\arc{AB}$ tal que $\mathcal{P}\subset\mathcal{Q}$, entonces diremos que $\mathcal{Q}$ es un refinamiento de $\mathcal{P}$.
\end{defi}

Sean $\mathcal{P}$ y $\mathcal{Q}$ particiones de $\arc{AB}$ y denotemos por $L(\mathcal{P})$ y $L(\mathcal{Q})$ a las longitudes de las poligonales asociadas. Nuestro siguiente objetivo consiste en mostrar que $|L(\mathcal{P})-L(\mathcal{Q})|$ puede hacerse arbitrariamente peque\~no, siempre que tanto $\|\mathcal{P}\|$ como $\|\mathcal{Q}\|$ sean suficientemente peque\~nos. En \'ultima instancia, aplicando este resultado a la sucesi\'on de particiones que dan origen a las poligonales de la subsecci\'on 1.2, y a otra sucesi\'on $\{\mathcal{Q}_m\}$ cualquiera verificando la condici\'on $\eqref{norma}$, concluiremos que 
$$L(\mathcal{Q}_m)\overset{m\to\infty}{\longrightarrow}\lim L_m,$$
donde $\{L_m\}$ es la sucesi\'on usada en las p\'aginas anteriores.\\
\\
La siguiente proposici\'on nos ayudar\'a a lograr lo anterior usando refinamientos comunes a $\mathcal{P}$ y $\mathcal{Q}$.

\begin{prop}
Si $\mathcal{P}$ es una partici\'on cualquiera de $\arc{AB}$ y $\mathcal{P}'$ es un refinamiento de $\mathcal{P}$, entonces
$$|L(\mathcal{P}')-L(\mathcal{P})|\le\frac{\ell^{(0)}}{\left(h^{(0)}\right)^2}\cdot\frac{\|\mathcal{P}\|^2}{4-\|\mathcal{P}\|^2},$$
siendo $\ell^{(0)}$ y $h^{(0)}$ las magnitudes usadas en las secciones anteriores.
\end{prop}

\begin{proof}
Lo primero que debemos hacer es notar que la relaci\'on $\eqref{cota}$ es v\'alida para la longitud de cualquier poligonal con v\'ertices en $\mathcal{C}$, y no solamente para las que usamos en la subsecci\'on 1.2. En efecto, sean $R$ y $S$ puntos en $\mathcal{C}$ y sean $\ell^{(0)}$ y $h^{(0)}$ las longitudes de la base y la altura en $O$ del tri\'angulo $\triangle ROS$. Dada cualquier poligonal de $n$ segmentos uniendo $R$ y $S$ con v\'ertices en $\arc{RS}$, la Proposici\'on 1 dice, usando la misma notaci\'on, que
$$\sum_{i=1}^n\ell_i\le\frac{\ell ^{(0)}}{h^{(0)}\displaystyle \min_{1\le i\le n}h_i}.$$
Como vimos al probar la Proposici\'on 2, se cumple que $h^{(0)}\le\displaystyle\min_{1\le i\le n}h_i\Rightarrow \left(h^{(0)}\right)^2\le h^{(0)}\min_{1\le i\le n}h_i$, y por lo tanto 
\begin{align}\label{COTA}\sum_{i=1}^n\ell_i\le\frac{\ell^{(0)}}{\left(h^{(0)}\right)^2}.\end{align}

\noindent Volviendo a lo principal, escribamos $\mathcal{P}=\{P_0,P_1,\ldots ,P_n\}$, y sean $\ell_i$ y $h_i$ las magnitudes de la base y la altura en $O$, respectivamente, del tri\'angulo $\triangle P_{i-1}OP_i$ (como en la Proposici\'on 1). Como $\mathcal{P'}$ es un refinamiento, cada punto de $\mathcal{P}$ coincide con alg\'un punto de $\mathcal{P'}$, de modo que por cada segmento $\overline{P_{i-1}P_i}$ de la poligonal asociada a $\mathcal{P}$ hay uno o m\'as segmentos de la poligonal asociada a $\mathcal{P'}$ (para cada $i$, hay una porci\'on de la poligonal proveniente de $\mathcal{P'}$ que una $P_{i-1}$ con $P_i$). Sea entonces $S_i$ la longitud de la porci\'on de poligonal asociada $\mathcal{P'}$ que le corresponde a $\overline{P_{i-1}P_i}$.\\
\\
Al aplicar la relaci\'on \eqref{COTA} sobre esta porci\'on obtenemos que $\displaystyle S_i\le \frac{\ell_i}{h_i^2},$ lo que implica que
$$\sum_{i=1}^nS_i\le\sum_{i=1}^n\frac{\ell_i}{h_i^2}.$$
Observando con cuidado podemos notar que el miembro de la izquierda corresponde a $L(\mathcal{P'})$. Adem\'as $\sum \ell_i\le\sum S_i$ gracias a la primera desigualdad de la Proposici\'on 1, por lo tanto
$$|L(\mathcal{P'})-L(\mathcal{P})|=\sum_{i=1}^nS_i-\sum_{i=1}^n\ell_i\le \sum_{i=1}^n\left(\frac{\ell_i}{h_i^2}-\ell_i\right)=\sum_{i=1}^n\ell_i\left(\frac{1}{h_i^2}-1\right).$$
Para llegar a algo que dependa exclusivamente de $\|\mathcal{P}\|$ (despu\'es de todo, esto es lo que queremos) notemos que
$$\ell_i\le \|\mathcal{P}\|\Rightarrow h_i^2=1-\ell_i^2\ge 1-\left(\frac{\|\mathcal{P}\|}{2}\right)^2\Rightarrow\frac{1}{h_i^2}-1\le\frac{1}{1-\left(\frac{\|\mathcal{P}\|}{2}\right)^2}-1=\frac{\|\mathcal{P}\|^2}{4-\|\mathcal{P}\|^2},$$
donde la primera desigualdad es cierta por definici\'on, y la implicancia descansa en el hecho de que $\|\mathcal{P}\|\le \sqrt{2}$ (la partici\'on m\'as gruesa posible se obtiene al tomar solamente los puntos $A$ y $B$, estando ambos a distancia menor o igual que $\sqrt{2}$), y por lo tanto $\left(\frac{\|\mathcal{P}\|}{2}\right)^2\le 1$, de donde se sigue que $1-\left(\frac{\|\mathcal{P}\|}{2}\right)^2$ es positivo.\\
\\
Usando este resultado nos acercamos a lo que queremos:
$$\sum_{i=1}^n\ell_i\left(\frac{1}{h_i^2}-1\right)\le\sum_{i=1}^n\ell_i\cdot\frac{\|\mathcal{P}\|^2}{4-\|\mathcal{P}\|^2}=L(\mathcal{P})\cdot\frac{\|\mathcal{P}\|^2}{4-\|\mathcal{P}\|^2},$$
donde podemos aplicar una vez m\'as la relaci\'on \eqref{COTA} para concluir que
$$|L(\mathcal{P'})-L(\mathcal{P})|\le\frac{\ell^{(0)}}{\left(h^{(0)}\right)^2}\cdot\frac{\|\mathcal{P}\|^2}{4-\|\mathcal{P}\|^2}.$$
\end{proof}

El hecho de que la expresi\'on de la derecha tiende a cero cuando $\|\mathcal{P}\|$ lo hace nos permite responder la pregunta inicial de esta secci\'on de la siguiente manera. Sean $\{\mathcal{P}_m\}$ y $\{\mathcal{Q}_m\}$ dos sucesiones de particiones de $\mathcal C$ verificando la ecuaci\'on $\ref{norma}$. Para cada $m$ definamos $\mathcal{R}_m$ como la partici\'on obtenida al unir los puntos de las dos anteriores; $\mathcal{R}_m$ es claramente un refinamiento de ambas.\\
\\
Sea $\varepsilon>0$ arbitrario. Escribiendo $f(x)=\displaystyle\frac{\ell^{(0)}}{\left(h^{(0)}\right)^2}\cdot\frac{x^2}{4-x^2},$ la aritm\'etica de l\'imites indica que $f(x)\overset{x\to 0}{\longrightarrow}0$, por lo tanto existe $\delta>0$ tal que si $|x|<\delta$ entonces $f(x)<\displaystyle\frac{\varepsilon}{4}$.\\
\\
Dado que las normas de $\mathcal{P}_m$ y $\mathcal{Q}_m$ tienden a cero, existe un $m_0$ tal que $\|\mathcal{P}_m\|<\delta$ y $\|\mathcal{Q}_m\|<\delta$ para $m\ge m_0$, por lo tanto $f(\|\mathcal{P}_m\|),f(\|\mathcal{Q}_m\|)<\displaystyle\frac{\varepsilon}{4}$ para estos valores de $m$. Usando esto y la Proposici\'on 5 concluimos que
\begin{align*}|L(\mathcal{P}_m)-L(\mathcal{Q}_m)|&=|L(\mathcal{P}_m)-L(\mathcal{R}_m)+L(\mathcal{R}_m)-L(\mathcal{Q}_m)|\\
&\le |L(\mathcal{P}_m)-L(\mathcal{R}_m)|+|L(\mathcal{R}_m)-L(\mathcal{Q}_m)|\\
&\le f(\|\mathcal{P}_m\|)+f(\| \mathcal{Q}_m\|)<\frac{\varepsilon}{2},\ \ \forall m\ge m_0\end{align*}

En el procedimiento de la subsecci\'on 1.2, la norma de la $m$-\'esima partici\'on es menor o igual que $\frac{|AB|}{\sqrt{2}^m}$ para todo $m$, por lo tanto tiende a cero cuando $m$ tiende a infinito. As\'i, si $\mathcal{P}_m$ es tal partici\'on (teni\'endose entonces que $L(\mathcal{P}_m)=L_m$) y $\{\mathcal{Q}_m\}$ es una sucesi\'on cualquiera de particiones con norma tendiendo a cero, podemos concluir que $\displaystyle\lim_{m\to\infty}L(\mathcal{Q}_m)=\lim_{m\to\infty}L_m.$ En efecto, por la convergencia de $L_m$ existe $N$ tal que $|L_m-\lim L_m|<\varepsilon/2$ cuando $m\ge N$. Adem\'as por el resultado de arriba hay alg\'un $M$ tal que $|L(\mathcal{Q}_m)-L_m|<\varepsilon/2$ cuando $m\ge M$, de modo que
$$|L(\mathcal{Q}_m)-\lim L_m|\le |L(\mathcal{Q}_m)-L_m|+|L_m-\lim L_m|<\varepsilon,\ \ \forall m\ge \max\{M,N\},$$
lo que prueba que la sucesi\'on de longitudes $\{L(\mathcal{Q}_m)\}$ converge y su l\'imite es igual a $\displaystyle\lim_{m\to\infty}L_m$.\\
\\
En resumen, hemos probado que si
\begin{enumerate}
\item $\{\mathcal{Q}_m\}$ es una sucesi\'on de particiones de $\arc{AB}$ tal que $\|\mathcal{Q}_m\|\overset{m\to\infty}{\longrightarrow}0$,
\item a cada partici\'on $\mathcal{Q}_m$ asociamos una poligonal ordenada y llamamos $L(\mathcal{Q}_m)$ a la longitud de esta,
\item definimos $|\arc{AB}|:=\displaystyle\lim_{m\to\infty}L(\mathcal{Q}_m)$,
\end{enumerate}
entonces la magnitud $|\arc{AB}|$ as\'i definida es necesariamente igual a aquella definida en la secci\'on 1.2, a saber, $\displaystyle\lim_{m\to\infty}L_m$. Esto muestra que la relaci\'on $|\arc{AB}|=2|S(A,B)|$ se verifica a\'un cuando $|\arc{AB}|$ se define mediante $\mathcal{Q}_m$.\\
\\
Con esto hemos finalizado el trabajo que nos propusimos en un comienzo, pero antes de terminar queremos hacer una peque\~na reflexi\'on acerca de algo que qued\'o en el tintero. En el segundo p\'arrafo de la demostraci\'on de la Proposici\'on 4 afirmamos que el \'area del sector $S(Y,Y_0)$ viene dada por $|g(y)-g(y_0)|$, es decir, asumimos que la suma de las \'areas de dos sectores circulares adyacentes es igual al \'area del sector circular mayor formado por ambos (asumimos que se verifica la \textit{aditividad} para las \'areas de sectores circulares). Si bien este hecho podr\'ia demostrarse laboriosamente partiendo de la aditividad de las \'areas de pol\'igonos, queremos mostrar que tambi\'en es una consecuencia de lo que ya probamos en esta secci\'on.\\
\\
Para esto usaremos nuevamente el hecho de que $|\arc{AB}|=2|S(A,B)|$, y probaremos que la aditividad se verifica para las longitudes de arcos de circunferencia. Una vez demostrado esto \'ultimo, la aditividad de las \'areas ser\'a una consecuencia directa de la relaci\'on de proporcionalidad anterior. Antes de continuar, una vez m\'as, debemos adoptar un punto de partida; en este caso asumiremos que la aditividad de longitudes es v\'alida para poligonales, es decir, supondremos que la suma de las longitudes de dos poligonales que coinciden en uno de sus extremos es igual a la longitud de la poligonal mayor formada por ambas.\\
\\
Sean $A,B$ en $\mathcal{C}$ y sea $M$ cualquier punto en $\mathcal C$ entre $A$ y $B$. Sea $\{\mathcal{P}_m\}$ una sucesi\'on de particiones de $\arc{AB}$ verificando \eqref{norma}, de modo que $\displaystyle\lim_{m\to\infty}L(\mathcal {P}_m)=|\arc{AB}|$, y para cada $m$ sea $\mathcal{P}^*_m$ la partici\'on obtenida al incluir el punto $M$ en $\mathcal{P}_m$. La primera observaci\'on que debemos hacer consiste en que la sucesi\'on $\{\mathcal{P}^*_m\}$ verifica \eqref{norma}, y por lo tanto $\displaystyle\lim_{m\to\infty}L(\mathcal{P}^*_m)=|\arc{AB}|$. Lo siguiente consiste en notar que, para cada $m$, la partici\'on $\mathcal{P}^*_m$ puede verse como la uni\'on de dos colecciones, una que da forma a una partici\'on de $\arc{AM}$, digamos $\mathcal{Q}^*_m$, y otra que forma una partici\'on de $\arc{MB}$, a la que podemos llamar $\mathcal{R}^*_m$. Claramente las sucesiones $\{\mathcal{Q}^*_m\}$ y $\{\mathcal{R}^*_m\}$ tambi\'en verifican \eqref{norma}, por lo tanto las sucesiones de longitudes $\{L(\mathcal{Q}^*_m)\}$ y $\{L(\mathcal{R}^*_m)\}$ convergen a $|\arc{AM}|$ y $|\arc{MB}|$ respectivamente. Adem\'as nuestra premisa de arriba nos lleva directamente a la igualdad $L(\mathcal{P}^*_m)=L(\mathcal{Q}^*_m)+L(\mathcal{R}^*_m)$, v\'alida para todo $m$ natural, de donde se sigue que
$$|\arc{AB}|=\lim_{m\to\infty} L(\mathcal{P}^*_m)=\lim_{m\to\infty} (L(\mathcal{Q}^*_m)+L(\mathcal{R}^*_m))=\lim_{m\to\infty}L(\mathcal{Q}^*_m)+\lim_{m\to\infty}L(\mathcal{R}^*_m)=|\arc{AM}|+|\arc{MB}|.$$
Como dijimos antes, esto garantiza la aditividad de las \'areas de sectores circulares. En efecto, si $A,M,B$ son puntos de $\mathcal{C}$ con $M$ entre $A$ y $B$, tenemos que
$$|S(A,B)|=\frac{1}{2}|\arc{AB}|=\frac{1}{2}(|\arc{AM}|+|\arc{MB}|)=|S(A,M)|+|S(M,B)|.$$

\section{Agradecimientos}
Quiero dar las gracias a Duvan Henao, profesor de la Facultad de Matem\'aticas, quien fue mi gu\'ia en la escritura de este trabajo. Gracias por motivar las reflexiones iniciales sobre este problema, por darme posteriormente la privilegiada oportunidad de participar en esto, por las incontables ocasiones en que estuvo dispuesto a supervisar el progreso del texto, y por las enriquecedoras discusiones que surgieron en el camino.\\
\\
Tambi\'en agradezco al profesor Mario Ponce por aportar con el ejemplo de la poligonal inapropiada mostrada en el resumen, y al profesor V\'ictor Cort\'es por revisar este trabajo y aportar con sugerencias aplicadas en la versi\'on final del texto.

\end{document}